\documentclass[12pt]{amsart}

\usepackage{color}

\usepackage{amsmath,amssymb,setspace,nicefrac, yhmath, amscd,eucal}
\usepackage[active]{srcltx}
\usepackage[colorlinks, linkcolor=red, citecolor=blue, urlcolor=blue, hypertexnames=true]{hyperref}
\usepackage{amsrefs, mathrsfs}

\allowdisplaybreaks
\usepackage[all]{xy}

\newcommand{\R}{{\mathbb R}}

\newcommand{\Z}{{\mathbb Z}}
\newcommand{\N}{{\mathbb N}}

\newcommand{\T}{{\mathbb T}}
\newcommand{\C}{{\mathbb C}}

\newcommand{\cT}{{\mathcal T}}

\newcommand{\diam}{{\rm diam}}
\newcommand{\Opt}{{\rm Opt}}
\newcommand{\Act}{{\rm Act}}
\newcommand{\Isom}{{\rm Isom}}

\newtheorem{thm}{Theorem}[section]
\newtheorem{cor}[thm]{Corollary}
\newtheorem{lem}[thm]{Lemma}

\newtheorem{definition}[thm]{Definition}
\newtheorem{example}[thm]{Example}
\newtheorem{remark}[thm]{Remark}

\theoremstyle{definition}

\keywords{Gromov-Hausdorff distance, Wasserstein space, group action, invariant measure, amenable group, stability, expansiveness, pseudo-orbit tracing property}
\subjclass[2010]{37C85, 37C40, 37C75}

\begin{document}
\title[Gromov-Hausdorff distances for dynamical systems]{Gromov-Hausdorff distances for dynamical systems}

\author{Nhan-Phu Chung}
\address{Nhan-Phu Chung, Department of Mathematics, Sungkyunkwan University, 2066 Seobu-ro, Jangan-gu, Suwon-si, Gyeonggi-do, Korea 16419.} 
\email{phuchung@skku.edu;phuchung82@gmail.com}
\date{\today}
\maketitle
\begin{abstract}
 We study equivariant Gromov-Hausdorff distances for general actions which are not necessarily isometric as Fukaya introduced. We prove that if an action is expansive and has the pseudo-orbit tracing property then it is stable under our adapted equivariant Gromov-Hausdorff topology. Finally, using Lott and Villani's ideas of optimal transport, we investigate equivariant Gromov-Hausdorff convergence for actions of locally compact amenable groups on Wasserstein spaces.         
\end{abstract}
\section{Introduction}
Gromov-Hausdorff distance on spaces of metric spaces was introduced by Gromov in his pioneering work in 1981 \cite{Gromov}. Gromov used it to prove his celebrated theorem: a finitely generated group is virtually nilpotent if and only it has polynomial growth. After that Gromov-Hausdorff distance has been extensively used to study  in convergence and collapsing of Riemannian manifolds by Cheeger, Colding, Fukaya, Gromov, and Yamaguchi \cite{CC,CFG,Fu86, Fu88, Fu90,FY}. In particular, in 1980s-1990s \cite{Fu86, Fu88, Fu90,FY}, Fukaya introduced the notion of equivariant Gromov-Hausdorff convergence for isometric actions of topological groups on Riemannian manifolds to study collapsing of Riemannian under bounded curvature and diameter, and fundamental groups of almost negatively curved manifolds. Here we would like to study distances on the space of continuous actions, which may be not isometric, a systematic study of adapted versions of equivariant Gromov-Hausdorff distance would be needed.
\iffalse
On the other hand, from Perelman's stability theorem in Alexandrov geometry, we know that for every $k\in\R$, $n\in \N$, every compact Alexandrov $n$-space $X$ of curvature $\geq k$ is stable under the Gromov-Hausdorff topology in the sense that there exists $\varepsilon>0$ such that for every compact Alexandrov $n$-space $Y$ of curvature $\geq k$ with $d_{GH}(X,Y)<\varepsilon$ is indeed homeomorphic to $X$, and every $\varepsilon$-Gromov-Hausdorff approximation can be approximated by a homeomorphism \cite{Perelman,Kapovitch}. In 2012, Rong and Xu investigated further this idea to study stability of exponential Lipschitz and co-Lipschitz maps in Gromov-Hausdorff topology \cite{RX}. 
\fi

Recently, Arbieto and Morales used techniques of \cite{Walters78} to establish stability under Gromov-Hausdorff topology for expansive maps having pseudo-orbit tracing property \cite{AM}. After that, combining ideas of \cite{AM} and \cite{ChungLee}, Arbieto and Morales' result has been extended for expansive actions of finitely generated groups with POTP in \cite{DLM} and \cite{KDD}. In another side, several stability results of equivariant Gromov-Hausdorff topology also have been proved for certain isometric actions \cite{Fu88,Harvey16,Harvey17}.
In this paper, we establish a result of stability under equivariant Gromov-Hausdorff distance for non-isometric actions of countable groups $G$ and $H$. More precisely, we prove 
\begin{thm}
\label{T-expansive and POTP implies GH-stability for two groups}
If an action $\alpha$ of a countable group $G$ on a proper metric space $(X,d_X)$ is expansive and satisfies POTP then it is strongly GH-stable. More precisely,
\begin{enumerate}
\item if $c>0$ is an expansive constant of $\alpha$ then for every $0<\varepsilon<c$ there exists $\delta>0$ such that if $\beta$ is a continuous action of a topological group $H$ on a metric space $(Y,d_Y)$ with $d_{GH,1}(\alpha,\beta)<\delta$ then there exist an $\varepsilon$-isometry $h:Y\to X$ and a homomorphism $\rho:G\to H$ satisfying $\alpha_g\circ h=h\circ \beta_{\rho(g)}$ for every $g\in G$. 
\item if furthermore $Y$ is compact then the $\varepsilon$-isometry $h$ can be chosen to be continuous.
\end{enumerate} 
\end{thm} 

Lastly, using ideas of Lott and Villani in \cite{LV}, we establish two equivariant Gromov-Hausdorff convergence results for induced actions on Wassertein spaces $(P_p(X),W_p)$ of continuous actions of topological groups on compact metric spaces $X$. %The first result in this direction is for general continuous actions of topological groups.  
\begin{thm}
\label{T-equivariant Gromov-Hausdorff of actions on Wasserstein spaces}
Let $\{\alpha_n\}$ be a sequence of continuous actions of a topological group $G$ on compact metric spaces $\{(X_n,d_n)\}$ and for every $p\geq 1$, let $(\alpha_n)_*$ be the induced action of $\alpha_n$ on $P_p(X_n)$ for every $n\in \N$. If $\lim_{n\to \infty}d_{mGH}(\alpha_n,\alpha)=0$ for some action $\alpha$ of $G$ on a compact metric space $(X,d)$ then $\lim_{n\to \infty}d_{mGH}((\alpha_n)_*,\alpha_*)=0$.
\end{thm}

\begin{thm}
\label{T-Gromov-Hausdorff convergence of invariant measures for amenable groups}
Let $\{\alpha_n\}$ be a sequence of isometric actions of a locally compact, $\sigma$-compact amenable group $G$ on compact metric spaces $\{(X_n,d_n)\}$ and for every $p\geq 1$, let $(\alpha_n)_*$ be the induced action of $\alpha_n$ on $P_p(X_n)$ for every $n\in \N$. If $d_{GH}(\alpha_n,\alpha)\to0$ as $n\to \infty$ for some action $\alpha$ of $G$ on a compact metric space $(X,d)$ then $\alpha$ is an isometric action and
$$\lim_{n\to\infty}d_{GH}(P_p^G(X_n),P_p^G(X))= 0.$$ 
\end{thm}
Here $P_p^G(X)$ is the space of $G$-invariant measures on $X$. As a consequence, we get
\begin{cor}
\label{C-uniquely ergodic}
Let $G$ be a $\sigma$-compact, locally compact amenable group and let $\{\alpha_n\}$ be a sequence of isometric actions of $G$ on compact metric spaces $\{(X_n,d_n)\}$. Assume that $\lim_{n\to \infty}d_{GH}(\alpha_n,\alpha)=0$ for some action $\alpha$ of $G$ on a compact metric space $(X,d)$. If $\alpha$ is uniquely ergodic then $\lim_{n\to \infty}\diam (P_p^G(X_n))=0$ for every $p\geq 1$. 
\end{cor}

The paper is organized as following. In Section 2, we review Gromov-Hausdorff distance for the space of all metric spaces, and Wasserstein spaces.  In Section 3, we define our adapted definitions for equivariant Gromov-Hausdorff distances for group actions and present their basic properties: the subsections 3.1 and 3.2 are for the cases of actions of a topological group $G$, and actions of topological groups $G$ and $H$, respectively. We also prove in this section theorem \ref{T-expansive and POTP implies GH-stability for two groups} and illustrate examples that it can apply. Finally, in Section 4, we will explain whenever Gromov-Hausdorff approximations can be approximated by measurable ones and then prove theorems \ref{T-equivariant Gromov-Hausdorff of actions on Wasserstein spaces}, \ref{T-Gromov-Hausdorff convergence of invariant measures for amenable groups}.  

\textbf{Acknowledgements:} Part of this paper was carried out when the author visited University of Science, Vietnam National University at Hochiminh city on July 2018. I am grateful to Dang Duc Trong for his warm hospitality. The author was partially supported by the National Research Foundation of Korea (NRF) grants funded by the Korea government(MSIP) (No. NRF- 2016R1A5A1008055, No. NRF-2016R1D1A1B03931922 and No. NRF-2019R1C1C1007107). We thank the anonymous referees for their useful comments which vastly improve the paper.
\section{Preliminaries}

First, we recall definition of Gromov-Hausdorff distance and its basic properties. For more details, the readers are referred to \cite{BBI,Rong,Shioya}. Let $(X,d)$ be a metric space and let $\varepsilon>0$. A subset $S$ of $X$ is called an \textit{$\varepsilon$-net} if $B'_\varepsilon(S)=X$, where $B'_\varepsilon(S):=\{y\in X: \mbox{ there exists } s\in S, d(y,s)\leq \varepsilon\}$. 
%i.e. for every $x\in X$, there exists $s\in S$ such that $d(x,s)\leq\varepsilon$.
\begin{definition}
Let $(Z,d)$ be a metric space and $X,Y$ be subsets of $Z$. The Hausdorff distance between $X$ and $Y$, denoted by $d_H(X,Y)$, is the infimum of $\varepsilon>0$ such that $X\subset B'_\varepsilon(Y)$ and $Y\subset B'_\varepsilon(X)$.
\end{definition} 
\begin{definition}
Let $X$ and $Y$ be metric spaces. The Gromov-Hausdorff (GH) distance between $X$ and $Y$, denoted by $d_{GH}(X,Y)$, is defined as the infimum of $r>0$ such that there exist a metric space $(Z,d)$ and its subspaces $X'$ and $Y'$ being isometric to $X$ and $Y$ respectively such that $d_H(X',Y')<r$.
\end{definition}
The GH-distance $d_{GH}$ is a metric on the space of all isometry classes of compact metric spaces \cite[Theorem 7.3.30]{BBI}.
\begin{definition}
Let $(X,d_X)$ and $(Y,d_Y)$ be metric spaces and let $\varepsilon>0$. An $\varepsilon$-isometric map between $X$ and $Y$ is a map $f:X\to Y$ satisfying $$|d_Y(f(x_1),f(x_2))-d_X(x_1,x_2)|\leq \varepsilon \mbox{ for every } x_1,x_2\in X.$$
We call a map $f:X\to Y$ is an $\varepsilon$-isometry if it is an $\varepsilon$-isometric map and $Y=B'_\varepsilon(f(X))$. In this case the map $f$ is also called an $\varepsilon$-GH approximation from $X$ to $Y$. 
\end{definition} 
\begin{definition}
\label{D-inverse GHA}
An $\varepsilon$-GH approximation $f:X\to Y$ has an approximation inverse $f':Y\to X$ constructed as following. Given $y\in Y$, we choose $x\in X$ such that $d_Y(f(x),y)\leq \varepsilon$ and we define $f'(y):=x$. Then $f':Y\to X$ is a $3\varepsilon$-GH approximation. From the construction of $f'$ it is clear that $\sup_{x\in X}d_X(x,(f'\circ f)(x))\leq 2\varepsilon$ and $\sup_{y\in Y}d_Y(y,(f\circ f')(y))\leq \varepsilon$.
\end{definition}
For every $\varepsilon>0$, if $d_{GH}(X,Y)<\varepsilon$ then there exists a $2\varepsilon$-GH approximation from $X$ to $Y$; and if there exists an $\varepsilon$-GH approximation $f:X\to Y$ then $d_{GH}(X,Y)<2\varepsilon$ \cite[Corollary 7.3.28]{BBI}.
\begin{definition}
Let $X$ and $Y$ be metric spaces. We define an alternative GH-distance between $X$ and $Y$, denoted by $\hat{d}_{GH}(X,Y)$, as the following.
$$\hat{d}_{GH}(X,Y):=\inf\{\varepsilon>0: \mbox{ there are }\varepsilon-GH \mbox{ approximations } f:X\to Y, g:Y\to X\}$$
if the infimum exists, and $\hat{d}_{GH}(X,Y)$ is $\infty$ if the infimum does not exist.
\end{definition}
From \cite[Lemma 1.3.4]{Rong}, we know that $\frac{2}{3}d_{GH}\leq \hat{d}_{GH}\leq 2d_{GH}$.

\begin{definition}
Let $X$ and $Y$ be metric spaces and let $\varepsilon,\delta>0$. We say that $X$ and $Y$ are $(\varepsilon,\delta)$-approximations of each other if there exist an $\varepsilon$-net $\{x_1,\cdots,x_m\}$ in $X$ and an $\varepsilon$-net $\{y_1,\cdots,y_m\}$ in $Y$ satisfying
$$|d_X(x_i,x_j)-d_Y(y_i,y_j)|<\delta, \mbox{ for every } 1\leq i,j\leq m.$$
\end{definition}
From the proof of \cite[Proposition 7.4.11]{BBI} we get the following lemma
\begin{lem}
\label{L-finite net implies closeness of GH}
Let $X$ and $Y$ be metric spaces. If $X$ and $Y$ are $(\varepsilon,\delta)$-approximations of each other then $d_{GH}(X,Y)<2\varepsilon+\delta$.
\end{lem}
%Now we recall the equivariant GH-distance introduced by Fukaya \cite{Fu86,Fu88,Fu90,FY}. Let $\alpha$ and $\beta$ be isometric actions of topological groups $G$ and $H$ on 
Now, we review Wasserstein spaces and optimal transport. The standard references for them are \cite{Villani03,Villani09}. Let $(X,d)$ be a metric space. For every $p\geq 1$, we denote by $P_p(X)$ the set of all probability Borel measures $\mu$ satisfying that there exists for some (and therefore for every) $x_0\in X$ such that $\int_Xd^p(x,x_0)d\mu(x)<\infty$. It is clear that if $X$ is bounded then $P_p(X)$ coincides with $P(X)$, the set of all probability Borel measures of $X$. For every probability Borel measures $\mu,\nu$ on $X$, we denote by $\prod(\mu,\nu)$ the set of all probability Borel measures on $X\times X$ with marginals $\mu$ and $\nu$. This means that $\pi\in \prod(\mu,\nu)$ if and only if $\pi$ is a Borel probability measure satisfying 
$$\pi(A\times X)=\mu(A), \pi(X\times B)=\nu(B),$$  for every Borel subsets $A,B$ of $X$.

 For every $p> 0$,  every $\mu,\nu\in P_p(X)$, and $\pi\in \prod(\mu,\nu)$, we define $$I_p(\pi)=\int_{X\times X}d^p(x_1,x_2)d\pi(x_1,x_2),$$
and then define the map $W_p$ on $P_p(X)\times P_p(X)$ by 
$W_p(\mu,\nu):=\cT_p^{1/p}(\mu,\nu)$, where $\cT_p(\mu,\nu):=\inf_{\pi\in \prod(\mu,\nu)}I_p(\pi)$ for every $\mu,\nu\in P_p(X).$ The map $W_p$ defines a metric on $P_p(X)$ \cite[Theorem 7.3]{Villani03}. If $X$ is compact then $P_p(X)$ is also compact \cite[Remark 6.19]{Villani09}.

For every $\mu,\nu\in P_p(X)$, we denote by $\Opt_p(\mu,\nu)$ the set of all $\pi_0\in \prod(\mu,\nu)$ such that $I_p(\pi_0)=\cT_p(\mu,\nu)$. 
If $X$ is a Polish space endowed with a metric $d$, i.e. the space $(X,d)$ is complete and separable, then $\Opt_p(\mu,\nu)\neq \varnothing$ for every $\mu,\nu\in P_p(X)$ \cite[Theorem 1.3]{Villani03}. 

Let $(X,d_X)$ and $(Y,d_Y)$ be metric spaces and $\varphi:X\to Y$ be a Borel map. Then we have the induced map $\varphi_*:P(X)\to P(Y),\mu\mapsto \varphi_*\mu$, where $ \varphi_*\mu(A):=\mu(\varphi^{-1}(A))$, for every Borel subset $A$ of $Y$. 

The following observation would be a basic fact in optimal transport, however we have not found a reference. Therefore, we give a simple proof for completeness.
\begin{lem}
\label{L-contraction of Wasserstein distance}
Let $(X_1,d_1)$ and $(X_2,d_2)$ be compact metric spaces and $f,g:X_1\to X_2$ be measurable maps. Then for every $p\geq 1$, every $\mu\in P(X_1)$, we have
$$W^p_p(f_*\mu,g_*\mu)\leq \int_{X_1}d_{X_2}^p(f(x_1),g(x_1))d\mu(x_1).$$
\end{lem}
\begin{proof}
Let $\pi:=(f\times g)_*\mu$ be the Borel probability measure on $X_2\times X_2$ defined by $$\pi(A\times B):=\mu(f^{-1}(A)\cap g^{-1}(B)),$$ for every Borel sets $A,B\subset X_2$. Then $\pi\in \prod(f_*\mu,g_*\mu)$ and for every non-negative measurable function $\zeta$ on $X_2\times X_2$, one has 
$$\int_{X_2\times X_2}\zeta(x_2,y_2)d\pi(x_2,y_2)=\int_{X_1}\zeta(f(x_1), g(x_1))d\mu(x_1).$$
Choose $\zeta=d_{X_2}^p$ we get the result.
\end{proof}
\section{Equivariant Gromov-Hausdorff distances for group actions}

\subsection{Equivariant Gromov-Hausdorff distances for actions of a topological group $G$}\hfill

For a topological group $G$ and a metric space $(X,d)$, we denote by $\Act(G,X)$ the space of all actions of $G$ on $(X,d)$. Before defining equivariant GH-distances, we recall the $C^0$ distance between two maps $f,g:(X,d)\to (X,d)$ as follows
$$d_{sup}(f,g):=\sup_{x\in X}d(f(x),g(x)).$$

Let $\alpha$ and  $\beta$ be actions of $G$ on metric spaces $(X,d_X)$ and $(Y,d_Y)$ respectively. Let $S$ be a generating set of $G$ and let $\varepsilon>0$. When referring to a map $f:X\to Y$ satisfying suitable properties related to the given actions of $G$, by abuse of notion we shall write $f:G\curvearrowright X\to G\curvearrowright Y$. We then say that a map $f:G\curvearrowright X\to G\curvearrowright Y$ is an $(\varepsilon,S)$-GH approximation (GHA) from $\alpha$ to $\beta$ if it is an $\varepsilon$-isometry satisfying that 
$d_{sup}(\beta_s\circ f,f\circ \alpha_s)\leq\varepsilon$ for every $s\in S$. If $f$ is furthermore Borel we say that $f$ is an $(\varepsilon,S)$-measurable GHA.
\begin{definition}
\label{D-GH distance for one group}
Let $\alpha$ and  $\beta$ be actions of $G$ on metric spaces $(X,d_X)$ and $(Y,d_Y)$ respectively. Let $S$ be a generating set of $G$. The equivariant GH-distances $d_{GH,S}$ and $d_{mGH,S}$ between $\alpha$ and $\beta$ with respect to $S$ are defined by 
\begin{align*}
d_{GH,S}(\alpha,\beta):=\inf\{\varepsilon>0: \mbox{ }&\exists (\varepsilon,S)-\mbox{ GHAs } f:G\curvearrowright X\to G\curvearrowright Y \\
&\mbox{ and } g:G\curvearrowright Y\to G\curvearrowright X \},
\end{align*}
\begin{align*}
d_{mGH,S}(\alpha,\beta):=\inf\{\varepsilon>0: \mbox{ }&\exists \varepsilon-\mbox{measurable GHAs } f:G\curvearrowright X\to G\curvearrowright Y, \\
&\mbox{ and } g:G\curvearrowright Y\to G\curvearrowright X \},
\end{align*}
if the above inf exist with the usual convention that $\inf \varnothing=+\infty$.
\end{definition} 
The definition of $d_{GH,S}$ was introduced by Abrieto and Morales \cite{AM} when $G$ is the semigroup $\N$ and $S=\{1\}$. After that, it has been extended for actions of a finitely generated group $G$ with $S$ is a finite generating set of $G$ in \cite{DLM} and \cite{KDD}. Note that when $S=G$, the definition of $d_{GH,S}(\alpha,\beta)$ coincides with the \cite[Definition 6.8]{Fu90}.
If $S=G$ we will write $d_{GH}(\alpha,\beta)$ and $d_{mGH}(\alpha,\beta)$ instead of $d_{GH,G}(\alpha,\beta)$ and $d_{mGH,G}(\alpha,\beta)$, respectively; and we also write $\varepsilon$-GHA for $(\varepsilon,G)$-GHA.
\begin{remark}
We shall see later in Lemma \ref{L-measurable eGHA} that for certain actions, given $\varepsilon>0$, we can replace an $\varepsilon$-GHA by a $D(\varepsilon)$-measurable GHA, where $D(\varepsilon)\to 0$ as $\varepsilon\to 0$. Therefore, in such cases, convergences in $d_{GH,S}$ and $d_{mGH,S}$ are the same.
\end{remark}
For $\alpha\in \Act(G,X)$ and $\beta\in \Act(G,Y)$, we say that $\alpha$ and $\beta$ are isometric if there exists an isometry $f:X\to Y$ such that $f\circ \alpha_g=\beta_g\circ f$ for every $g\in G$. 

Similar to \cite[Theorem 1]{AM} for the map case, here are basic properties of $d_{GH,S}$. 
\begin{lem}
\label{L-quasi GH metric on isometry classes of compact metric actions}
Let $\alpha$, $\beta$ and $\gamma$ be actions of a topological group $G$ on metric spaces $(X,d_X)$, $(Y,d_Y)$ and $(Z,d_Z)$ respectively. Let $S$ be a finitely generating set of $G$. The map $d_{GH,S}$ satisfies the following properties
\begin{enumerate}
\item $d_{GH,S}(\alpha,\beta)\geq 0$ and $d_{GH,S}(\alpha,\beta)=d_{GH,S}(\beta,\alpha)$;
\item $\hat{d}_{GH}(X,Y)\leq d_{GH,S}(\alpha,\beta)$ and $\hat{d}_{GH}(X,Y)=d_{GH,S}(\alpha^0,\beta^0)$, where $\alpha^0$ and $\beta^0$ are trivial actions of $G$ on $X$ and $Y$ respectively;
\item If $X=Y$ then $d_{GH,S}(\alpha,\beta)\leq d_S(\alpha,\beta)$, where $d_S(\alpha,\beta):=\sup_{s\in S,x\in X}d_X(\alpha_sx,\beta_sx)$;
\item If $X$ and $Y$ are bounded then $d_{GH,S}(\alpha,\beta)<\infty$;
\item $d_{GH,S}(\alpha,\beta)\leq 2(d_{GH,S}(\alpha,\gamma)+d_{GH,S}(\gamma,\beta));$
\item If $S$ is symmetric, i.e. $S=S^{-1}$, and $X,Y$ are compact then $d_{GH,S}(\alpha,\beta)=0$ if and only if $\alpha$ is isometric to $\beta$.
\end{enumerate}
\end{lem}
\begin{proof}
(1), (2), (3) and (4) are clear from the definitions.

(5) If one of $d_{GH,S}(\alpha,\gamma), d_{GH,S}(\gamma,\beta)$ is $\infty$ then we are done. Now we assume that $d_{GH,S}(\alpha,\gamma)<\infty$ and $d_{GH,S}(\gamma,\beta)<\infty$.  This case is proved in \cite[Theorem 4.1]{KDD}.

(6) Suppose that there exists an isometry $f:X\to Y$ such that $f\circ \alpha_g=\beta_g\circ f$ for every $g\in G$. Then $f^{-1}:Y\to X$ is also an isometry and therefore for every $\varepsilon>0$ $f,f^{-1}$ are $\varepsilon$-isometries satisfying $d_{sup}(\alpha_g\circ f^{-1},f^{-1}\circ \beta_g)=d_{sup}(f\circ \alpha_g,\beta_g\circ f)=0$ for all $g\in G$. Thus, $d_{GH,S}(\alpha,\beta)=0$.

Now suppose that $d_{GH,S}(\alpha,\beta)=0$. Then there exists a sequence of $\frac{1}{n}$-isometries $f_n:X\to Y$ and $g_n:Y\to X$ such that for every $n\in \N$,
$$\sup_{s\in S}\{d_{sup}(\alpha_s\circ g_n,g_n\circ \beta_s),d_{sup}(\beta_s\circ f_n,f_n\circ \alpha_s)\}\leq\frac{1}{n}.$$

 As $X$ is separable we can find a countable dense subset $A=\{a_n\}$ of $X$. Since $Y$ is compact, there exists a subsequence $\{f_{n_1}\}$ of $\{f_n\}$ such that $f_{n_1}(a_1)$ converges to $f(a_1)\in Y$, and we can assume that $d_Y(f_{n_1}(a_1),f(a_1))<1$ for every $n_1$. Similarly for $a_2$ and $\{f_{n_1}\}$ we can get a subsequence $f_{n_2}(a_2)$ converging to $f(a_2)\in Y$ and $d_Y(f_{n_2}(a_2),f(a_2))<1/2$ for every $n_2$. Repeating this process we get a subsequence of $\{f_n\}$ which we still denoted by $\{f_n\}$ such that for every $i\in \N$, $f(a_i):=\lim_{n\to \infty}f_{n}(a_i)$ and $d_Y(f_n(a_i),f(a_i))<1/i$ for every $n$. Therefore, we have the map $f:A\to Y$ defined by $f(a):=\lim_{n\to \infty}f_{n}(a)$ for every $a\in A$ , and furthermore $f_n\to f$ uniformly. On the other hand, for every $a_i,a_j\in A$ and $n\in \N$, one has 
 $$d_X(a_i,a_j)-\frac{1}{n}\leq d_Y(f_{n}(a_i),f_{n}(a_j))\leq d_X(a_i,a_j)+\frac{1}{n}.$$
 
 Hence, we get $d_X(a_i,a_j)=d_Y(f(a_i),f(a_j))$ for every $a_i,a_j\in A$. Let $x\in X$. As $A$ is dense in $X$, there exists a sequence $\{x_n\}$ in $A$ such that $x_n$ converges to $x$. It is a Cauchy sequence and hence $\{f(x_n)\}$ is also a Cauchy sequence and therefore has a limit $f(x):=y$.
Then the extension map $f:X\to Y$ also satisfies $d_X(x_1,x_2)=d_Y(f(x_1),f(x_2))$, for every $x_1,x_2\in X$. As $f_n$ is an $\frac{1}{n}$-isometry for every $n$, we obtain that $f$ is onto and therefore $f$ is an isometry.

 Next we will prove that $\beta_g\circ f=f\circ \alpha_g$ for every $g\in G$. As $S$ is a symmetric generating set of $G$, it suffices to prove that $\beta_s\circ f=f\circ \alpha_s$ for every $s\in S$. Fix $s\in S$. Since $f_n\to f$ uniformly on $A$ and $f_n$ is an $\frac{1}{n}$-isometry for every $n\in \N$, one has $\lim_{n\to \infty}f_{n}(x)=f(x)$ for every $x\in X$. Therefore from $d_Y(\beta_s\circ f_{n}(x),f_{n}\circ \alpha_s(x))\leq\frac{1}{n}$ for every $n\in \N,x\in X$, we get $d_Y(\beta_s\circ f(x),f\circ \alpha_s(x))=0$. Hence $\beta_s\circ f=f\circ \alpha_s$.
 \end{proof}
 \begin{remark}
 After this paper has been finished, I received the preprint \cite{DLM} in which a similar result of Lemma \ref{L-quasi GH metric on isometry classes of compact metric actions} (6) also has been proved for the case $G$ is a finitely generated group. 
 \end{remark}
\begin{lem}
\label{L-closedness of isometric actions}
Let $G$ be a topological group and let $S$ be a symmetric generating set of $G$. Assume that for each $n$ we have an isometric action of $G$ on  a metric space $(X_n,d_{X_n})$. Assume further that there exists an action of $G$ on a metric space $(X,d)$ such that 
 $d_{GH,S}(\alpha_n,\alpha)\to 0$ as $n\to \infty$. Then $\alpha$ is an isometric action.
\end{lem}
\begin{proof}
From the assumptions, there exist a sequence $\varepsilon_n\to 0$ and $\varepsilon_n$-isometries $f_n:X\to X_n$ and $g_n:X_n\to X$ such that 
$$\sup_{s\in S}\{d_{sup}(\alpha_{n,s}\circ f_n,f_n\circ \alpha_s),d_{sup}(g_n\circ \alpha_{n,s},\alpha_s\circ g_n)\}<\varepsilon_n, \mbox{ } \forall n\in\N.$$
Then for every $s\in S$, $x_1,x_2\in X$, $n\in \N$, one has 
\begin{eqnarray*}
d_X(\alpha_s(x_1),\alpha_s(x_2))&\leq& \varepsilon_n+d_{X_n}(f_n\circ \alpha_s(x_1),f_n\circ \alpha_s(x_2))\\
&\leq& \varepsilon_n+d_{X_n}(f_n\circ \alpha_s(x_1),\alpha_{n,s}\circ f_n(x_1))+\\
&& d_{X_n}(\alpha_{n,s}\circ f_n(x_1),\alpha_{n,s}\circ f_n(x_2))+\\
&&d_{X_n}(\alpha_{n,s}\circ f_n(x_2),f_n\circ \alpha_s(x_2))\\
&\leq&4\varepsilon_n+d_X(x_1,x_2).
\end{eqnarray*}
Therefore, for every $s\in S$, $x_1,x_2\in X$, $d_X(\alpha_s(x_1),\alpha_s(x_2))\leq d_X(x_1,x_2)$. This means that for every $s\in S$, the map $\alpha_s$ is contracting. As $S$ is symmetric we obtain that for every $s\in S$, both $\alpha_s$ and $\alpha_{s^{-1}}$ are contracting and then they are isometries.
\end{proof}
Now let $G$ be a countable group. We recall the definitions of expansive actions and the pseudo-orbit tracing property (or shadowing property).
Let $\alpha$ be a continuous action of $G$ on a metric space $(X,d)$. 
\begin{definition}
For a subset $S$ of $G$ and $\delta>0$, a $(\delta,S) \mbox{ pseudo-orbit}$ of $\alpha$ is a sequence $\{x_g\}_{g\in G}$ such that $d(\alpha_s(x_g),x_{sg})<\delta$ for every $s\in S,g\in G$. 
\end{definition}
\begin{definition}
\label{D-POTP} (\cite[Definition 2.2]{Mey})
The action $\alpha$ has the \textit{pseudo-orbit tracing property} (POTP) if for every $\varepsilon>0$ there exist $\delta>0$ and a finite subset $S$ of $G$ such that every $(\delta,S)$ pseudo-orbit $\{x_g\}_{g\in G}$ is $\varepsilon$-traced by some point $x\in X$, i.e. $d(\alpha_g(x),x_g)<\varepsilon$ for every $g\in G$. 
\end{definition}
POTP was introduced firstly by Rufus Bowen \cite{Bowen75,Bowen75b} when $G=\Z$, and has been extended for $G=\Z^d$ \cite{PT,Oprocha} and in the case $G$ is a finitely generated group \cite{OT}. Note that the definition of POTP in \cite{OT} is a special case of Definition \ref{D-POTP} because if $G$ is generated by a finite set $S$ then $\alpha\in \Act(G,X)$ has POTP (with respect to S) in the sense of \cite{OT} if for every $\varepsilon>0$ there exists $\delta>0$ such that every $(\delta,S)$ pseudo-orbit $\{x_g\}_{g\in G}$ is $\varepsilon$-traced by some point $x\in X$. 
\begin{definition}
\label{D-expansivity}
The action $\alpha\in \Act(G,X)$ is \textit{expansive} if there exists an expansive constant $c>0$ such that for every $x\neq y\in X$, $\sup_{g\in G}d(\alpha_gx,\alpha_gy)>c$.
\end{definition}  
\begin{remark}
\label{R-unique traced of pseudo-orbit}
Let $\alpha$ be an expansive action of a countable group $G$ on a metric space $(X,d)$ with an expansive constant $c$. Let $\varepsilon<c/2$ and $\delta, S$ corresponds to $\varepsilon$ as in Definition \ref{D-POTP}. Then every $(\delta,S)$ pseudo-orbit of $\alpha$ is $\varepsilon$-traced by a unique point in $X$. 
\end{remark}
\begin{definition}
An action $\alpha$ of a topological group $G$ on a metric space $X$ is \textit{GH-stable} if for every $\varepsilon>0$, there is $\delta>0$ such that for every continuous action $\beta$ of $G$ on a metric space $Y$ with $d_{GH}(\alpha,\beta)<\delta$, there is a continuous $\varepsilon$-isometry $h:Y\to X$ such that $\alpha_g\circ h=h\circ \beta_g$ for every $g\in G$.
\end{definition}
 The first result of topological stability for maps which are expansive and have POTP was established by Walters in \cite{Walters78}. And in \cite{ChungLee}, Chung and Lee proved a group action version of Walters' topological stability result. On the other hand, in 2017, Arbieto and Morales used techniques of \cite{Walters78} to establish stability under GH-topology for expansive maps having pseudo-orbit tracing property \cite{AM}. After that, combining ideas of \cite{AM} and \cite{ChungLee}, a version of GH-stability for such actions of a finitely generated group $G$ has been established in \cite{DLM} and \cite{KDD}. Following the proofs of \cite{ChungLee} and \cite[Theorem 4.6]{KDD}, we can see that the finitely generating condition of $G$ is not necessary.
\begin{thm}
\label{T-expansive and POTP implies GH-stability for one group}
If an action $\alpha$ of a countable group $G$ on a proper metric space $(X,d)$ is expansive and satisfies POTP then it is topologically GH-stable with respect to $S$. More precisely,
\begin{enumerate}
\item if $c>0$ is an expansive constant of $\alpha$ then for every $0<\varepsilon<c$ there exists $\delta>0$ such that if $\beta$ is another continuous action of $G$ on a metric space $(Y,d_Y)$ with $d_{GH,S}(\alpha,\beta)<\delta$ then there exists an $\varepsilon$-isometry $h:Y\to X$ satisfying $\alpha_g\circ h=h\circ \beta_g$ for every $g\in G$. 
\item if furthermore $Y$ is compact then the $\varepsilon$-isometry $h$ can be chosen to be continuous.
\end{enumerate} 
\end{thm} 
As Theorem \ref{T-expansive and POTP implies GH-stability for one group} is a special case of Theorem \ref{T-expansive and POTP implies GH-stability for two groups}, we skip its proof now. Instead, we present an example to illustrate Theorem \ref{T-expansive and POTP implies GH-stability for one group}.
\begin{example}
\label{E-GH distance} Given two metric spaces $(M,d_M),(N,d_N)$, we equip the product space $M\times N$ with the standard metric $d$ by setting $$d((x_1,x_2),(y_1,y_2)):=\sqrt{d_M^2(x_1,y_1)+d_2^N(x_2,y_2)},$$ for every $(x_1,x_2),(y_1,y_2)\in M\times N$.

 For every $r>0$ we denote $S^1_r:=\{(x_1,x_2)\in \R^2:x_1^2+x_2^2=r^2\}$. We endow $S^1_r$ with the canonical metric $d_r$, i.e. for every $s_1,s_2\in S^1_r$, $d_r(s_1,s_2)$ is the length of the smallest arc connecting $s_1$ and $s_2$. When $r=1$ we write $S^1$ instead of $S^1_1$. We denote by $d_{\R/\Z}$ the canonical metric on $\R/\Z$ defined by $d_{\R/\Z}(t + \Z, s + \Z) := \min_{m\in \Z}|t-s-m|.$
Let $X=\T^2=\R/\Z\times \R/\Z$ be the torus and $d_X$ be its canonical product metric. Put $Y=S^1_{1/n}\times S^1_{1/n}\times X$ and endows it with the product metric $d_Y$. We define the map $h:Y\to X$ by $h(s_1,s_2,x):=x$ for every $s_1,s_2\in S^1_{1/n},x\in \T^2$. Let $A,B\in M_2(\R)$ be two invertible matrices such that $AB=BA$, $A$ does not have eigenvalues of modulus 1 and $B^m\neq I_2$ for every $m\in \N$, where $I_2$ is the $2\times 2$ identity matrix. For example, we choose 
$A=\left(\begin{array}{cc} 1 & 3\\ 2 & 4 \end{array}\right), B=\left(\begin{array}{cc} -3 & 3\\ 2 & 0 \end{array}\right)$. Then the group $G$ generated by $\{A,B\}$ is isomorphic to $\Z^2$. Let $\alpha$ be the natural action of $G$ on $\T^2$ and let $\gamma$ be an arbitrary action of $G$ on $S^1_{1/n}\times S^1_{1/n}$. Let $\beta$ be the product action of $G$ on $(Y,d_Y)$ induced from the actions $\gamma$ and $\alpha$. Then we will have 
$\alpha_g\circ h=h\circ \beta_g$ for every $g\in G$. On the other hand, for every $(s_1,s_2,x_1),(t_1,t_2,x_2)\in Y$ we have
\begin{eqnarray*}
d_Y((s_1,s_2,x_1), (t_1,t_2,x_1))&=&\sqrt{d_{1/n}^2(s_1,t_1)+d_{1/n}^2(s_2,t_2)+d_X^2(x_1,x_2)}\\
&\leq & \sqrt{(\pi/n)^2+(\pi/n)^2+d_X^2(x_1,x_2)}\\
&\leq & d_X(x_1,x_2)+\frac{\sqrt{2}\pi}{n}\\
&=&d_X(h(s_1,s_2,x_1),h(t_1,t_2,x_2))+\frac{\sqrt{2}\pi}{n}.
\end{eqnarray*} 
Therefore, the map $h:Y\to X$ is $\frac{\sqrt{2}\pi}{n}$-isometric. As $h$ is also surjective we get that $h$ is a $\frac{\sqrt{2}\pi}{n}$-isometry. Fix $\bar{s}\in S^1_{1/n},$ we define the map $f:X\to Y$ by $f(x):=(\bar{s},\bar{s},x)$ for every $x\in X$. Then $f$ is isometric. In addition to, for every $(s_1,s_2,x)\in Y$, we have $d_Y((s_1,s_2,x),f(x))\leq \frac{\sqrt{2}\pi}{n}$ and hence $B'_{\frac{\sqrt{2}\pi}{n}}(f(X))=Y$. Therefore $f$ is a $\frac{\sqrt{2}\pi}{n}$-isometry. On the other hand, for every $g\in G,x\in X$ we have
\begin{eqnarray*}
d_Y(\beta_g\circ f(x),f\circ \alpha_g(x))&=&d_Y(\beta_g(\bar{s},\bar{s},x),(\bar{s},\bar{s},\alpha_g(x)))\\
&=&d_Y((\gamma_g(\bar{s},\bar{s}),\alpha_g(x)),(\bar{s},\bar{s},\alpha_g(x)))\\
&\leq & \sqrt{2\diam^2 S^1_{1/n}} = \frac{\sqrt{2}\pi}{n}.
\end{eqnarray*}
Therefore $d_{GH}(\alpha,\beta)\leq  \frac{\sqrt{2}\pi}{n}$. By the choice of $A$ we get that $\alpha_A$ is expansive and hence $\alpha$ is expansive. As $\alpha_A$ also has POTP and $G$ is nilpotent, applying \cite[Lemma 2.13]{ChungLee} we get that the action $\alpha$ has POTP.

\end{example}
\begin{definition}
Let $\alpha_n$ be a sequence of actions of a topological group $G$ on compact metric spaces $(X_n,d_n)$. We say that $\{\alpha_n\}$ is \textit{equicontinuous} if for every $\varepsilon>0$, there exists $\delta>0$ such that for every $g\in G$, $n\in \N$,  and for every $x_n,y_n\in X_n$ with $d_n(x_n,y_n)<\delta$, one has $d_n(\alpha_{n,g}(x_n),\alpha_{n,g}(y_n))\leq \varepsilon$.
\end{definition}
Following the ideas in \cite[page 66]{Gromov}, \cite[Appendix]{GP} and \cite[Lemma 11.1.9]{Petersen} we obtain a compactness result for equicontinuous actions.
\begin{lem}
Let $\{(X_n,d_n)\}$ be a sequence of metric spaces such that $d_{GH}(X_n,X)\to 0$ for some compact metric space $(X,d)$. Then we can assume there exist sequences of $1/n$-isometry maps $f_n:X_n\to X$, $h_n:X\to X_n$ such that for every $x\in X,x_n\in X_n$, one has 
 $$d(x,f_n\circ h_n(x))\leq 1/n \mbox{ and } d_n(x_n,h_n\circ f_n(x_n))\leq 1/n.$$ 
 Assume that for each $n$ we have an action $\alpha_n$ of $G$ on $X_n$. Assume further that $\{\alpha_n\}$ is equicontinuous. Then there exist a subsequence $\{X_{n_k}\}$ of $\{X_n\}$ and an action $\alpha$ of $G$ on $X$ such that 
 $\alpha_{n_k}\to \alpha $ as $n_k\to \infty$ in the sense that $$\alpha_g(x)=\lim_{n_k\to\infty}f_{n_k}\circ \alpha_{n_k,g}\circ h_{n_k}(x),$$ for every $x\in X,g\in G$.
\end{lem}
\begin{proof}
Let $A=\{a_n\}$ be a countable and dense subset of $X$. Fix $g\in G$. As $X$ is compact there exists a subsequence $\{n_1\}$ such that $\lim_{n_1\to \infty}f_{n_1}\circ \alpha_{n_1,g}\circ h_{n_1}(a_1)$ converges to some point in $X$, denoted by $\alpha_g(a_1)$. Using a standard diagonal argument, we can assume that there exists a subsequence $\{n_k\}$ such that $\alpha_g(a_m)=\lim_{n_k\to \infty}f_{n_k}\circ \alpha_{n_k,g}\circ h_{n_k}(a_m)$, for every $a_m\in A$. 
 Then for every $i,j$, 
 \begin{eqnarray*}
 d(\alpha_g(a_i),\alpha_g(a_j))&=&\lim_{k\to \infty}d(f_{n_k}\circ \alpha_{n_k,g}\circ h_{n_k}(a_i), f_{n_k}\circ \alpha_{n_k,g}\circ h_{n_k}(a_j)).
 \end{eqnarray*}
 Since $f_n$ and $h_n$ are $1/n$-isometries and by equicontinuity of $\{\alpha_n\}$, we see that the map $\alpha_g:A\to X$ is uniformly continuous and $f_{n_k}\circ \alpha_{n_k,g}\circ h_{n_k}\to \alpha_g$ uniformly on $A$ as $k\to\infty$. Therefore, we can extend it to a continuous map $\alpha_g:X\to X$. 
 As $d(x,f_{n}\circ h_n(x))\leq 1/n$ and $d_n(y,f_{n}\circ h_n(y))\leq 1/n$ for every $n\in \N$, $x\in X, y\in X_n$, and $\alpha_n$ is a family of equicontinuous actions, for every $s,t\in G, a\in A$, we have 
 \begin{eqnarray*}
 \alpha_{s}\circ \alpha_t(a)&=&\lim_{k\to\infty}f_{n_k}\circ \alpha_{n_k,s}\circ h_{n_k}\circ f_{n_k}\circ \alpha_{n_k,t}\circ h_{n_k}(a)\\
 &=&\lim_{k\to\infty}f_{n_k}\circ \alpha_{n_k,s}\circ \alpha_{n_k,t}\circ h_{n_k}(a) \\
 &=&\lim_{k\to\infty}f_{n_k}\circ \alpha_{n_k,st}\circ h_{n_k}(a)\\
 &=&\alpha_{st}(a),
 \end{eqnarray*}
and $\alpha_{e_G}(a)=\lim_{k\to \infty}f_{n_k}\circ \alpha_{n_k,e_G}\circ h_{n_k}(a)=a$, where $e_G$ is the identity element of $G$. 
 Therefore $\alpha_{e_G}=Id_X$ and $\alpha_s\circ \alpha_t=\alpha_{st}$ for every $s,t\in G$. For every $g\in G$, the continuity of the map $\alpha_g: X\to X$ is clear. Hence $\alpha$ is an action of $G$ on $X$ we are looking for.
 \end{proof}
\iffalse
 \begin{remark}
 We do not know whether $d_{GH}(\alpha_{n_k},\alpha)$ converges to 0 or not as $k\to \infty$.
 \end{remark}
 \fi
 
\subsection{Equivariant Gromov-Hausdorff distance for actions of $G$ and $H$}\hfill

Let $G$ and $H$ be topological groups. Let $\alpha$ and $\beta$ be actions of $G$ and $H$ on metric spaces $(X,d_X)$ and $(Y,d_Y)$, respectively. For $\varepsilon>0$, an $\varepsilon$-equivariant Gromov-Hausdorff approximation/equivariant Fukaya-Gromov-Hausdorff approximation (abbreviated, respectively, eGHA/eFGHA) $(\rho,f):\alpha \to \beta$ is a couple of maps $f:Y\to X$, $\rho:G\to H$ such that $\rho$ is a homomorphism/$\rho$ is a general map (not necessarily a homomorphism), and $f$ is an $\varepsilon$-GH approximation satisfying
$\sup_{g\in G}d_{sup}(\alpha_g\circ f,f\circ \beta_{\rho(g)})\leq \varepsilon$. 

We define GH-distances of $\alpha$ and $\beta$ by setting
 \begin{align*}
 d_{GH,1}(\alpha,\beta):=\inf\{&\varepsilon>0: \exists\varepsilon-\mbox{eGHA }(\rho,f):\alpha\to \beta \}, \\
d_{GH,2}(\alpha,\beta):=\inf\{&\varepsilon>0: \exists \varepsilon-\mbox{eGHAs }(\rho,f):\alpha\to \beta
\mbox{ and }(\varphi,h):\beta\to \alpha  \},
\end{align*}
if the above infimum exist and $\infty$ otherwise. 
\begin{remark}
\begin{enumerate}
\item The distances $d_{GH,1}$ and $d_{GH,2}$ were introduced by Fukaya for isometric actions \cite{Fu86,Fu88,Fu90,FY} and in his definitions he does not require $\rho:G\to H$ be a homomorphism. To study stability of general continuous actions which may be non-isometric, we adapt definitions of Fukaya to $d_{GH,1}$ and $d_{GH,2}$ by adding the homomorphism property of $\rho$.
\item In the case $G=H$, we see that $d_{GH,2}(\alpha,\beta)\leq d_{GH}(\alpha,\beta)$ where $d_{GH}$ is the GH-distance $d_{GH,G}$ defined in Definition \ref{D-GH distance for one group}.
\end{enumerate}
\end{remark}

\begin{definition}
A continuous action of a topological group $G$ on a metric space $X$ is topologically free if for every $g\in G\setminus \{e_G\}$, the set $\{x\in X:gx\neq x\}$ is dense in $X$.
\end{definition}

For every metric space $(X,d)$, we denote by $\Isom(X)$ the group of all isometry $f:X\to X$, and endow it with the compact-open topology. We recall that a sequence $\{ f_n\}\subset \Isom(X)$ converges to $ f\in \Isom(X)$  in the compact-open topology if and only if for every compact subset $K$ of $X$, $ f_n$ converges uniformly to $ f$  on $K$.
Now we present several basic properties of $d_{GH,1}$ and $d_{GH,2}$. 
\begin{lem}
Let $\alpha,\beta,\gamma$ be actions of topological groups $G,H,K$ on metric spaces $(X,d_X),(Y,d_Y)$ and $(Z,d_Z)$ respectively. The maps $d_{GH,1}$ and $d_{GH,2}$ satisfy
\begin{enumerate}
\item $d_{GH,1}(\alpha,\beta)\geq 0$, $d_{GH,2}(\alpha,\beta)\geq 0$ and $d_{GH,2}(\alpha,\beta)=d_{GH,2}(\beta,\alpha)$;
\item  Assume that $G$ and $H$ are closed in $\Isom(X)$ and $\Isom(Y)$, respectively, and assume further that $\alpha$ and $\beta$ are topologically free isometric actions. Then $d_{GH,2}(\alpha,\beta)=0$ if and only if there exist an isometry $f:Y\to X$ and an isomorphism $\rho:G\to H$ such that $\alpha_g\circ f=f\circ \beta_{\rho(g)}$ for every $g\in G$;
\item Assume that $X,Y$ are compact and $H$ is closed in $\Isom(Y)$. Assume further that $\beta$ is a topologically free isometric action. Then $d_{GH,1}(\alpha,\beta)=0$ if and only if there exist an isometry map $f:Y\to X$ and a homomorphism $\rho:G\to H$ such that $\alpha_g\circ f=f\circ \beta_{\rho(g)}$ for every $g\in G$;
\item  $d_{GH,i}(\alpha,\beta)\leq 2(d_{GH,i}(\alpha,\gamma)+d_{GH,i}(\gamma,\beta)), \mbox{ for i=1,2 }.$
\end{enumerate}
\end{lem}
\begin{proof}
(1) is clear from the definitions.\\
(2) and (3) are proved in the proof of \cite[Proposition 1.5]{Fu86}.\\
(4) We only need to prove $d_{GH,1}(\alpha,\beta)\leq 2(d_{GH,1}(\alpha,\gamma)+d_{GH,1}(\gamma,\beta))$. The remaining case is similar.  If one of $d_{GH,1}(\alpha,\gamma), d_{GH,1}(\gamma,\beta)$ is $\infty$ then we are done. Now we assume that $d_{GH,1}(\alpha,\gamma)<\infty$ and $d_{GH,1}(\gamma,\beta)<\infty$. Fix $\varepsilon>0$. Then there exist an $\varepsilon_1$-isometry $f:Z\to X$, an $\varepsilon_2$-isometry $v:Y\to Z$, and homomorphisms $\rho:G\to K$, $\varphi:K\to H$ such that $\varepsilon_1<d_{GH,1}(\alpha,\gamma)+\varepsilon$, $\varepsilon_2<d_{GH,1}(\gamma,\beta)+\varepsilon$ and
  \begin{align*}
&\sup_{g\in G}\{d_{sup}(f\circ \gamma_{\rho(g)},\alpha_g\circ f)\}\leq\varepsilon_1, \\
&\sup_{k\in K}\{d_{sup}(v\circ \beta_{\varphi(k)},\gamma_k\circ v)\}\leq\varepsilon_2.
\end{align*}
Hence, for every $y\in Y$ and $g\in G$, one has
\begin{eqnarray*}
d_X(f\circ v\circ \beta_{\varphi\circ\rho(g)}(y),\alpha_g\circ f\circ v(y))&\leq & d_X(f\circ v\circ \beta_{\varphi\circ\rho(g)}(y),f\circ \gamma_{\rho(g)}\circ v(y))+\\
&&+d_X(f\circ \gamma_{\rho(g)}\circ v(y),\alpha_g\circ f\circ v(y))\\
&\leq& \varepsilon_1+d_Z(v\circ \beta_{\varphi\circ\rho(g)}(y),\gamma_{\rho(g)}\circ v(y))+\varepsilon_1\\
&\leq& 2\varepsilon_1+\varepsilon_2\\
&<& 2(\varepsilon_1+\varepsilon_2).  
\end{eqnarray*}
Thus, $\sup_{g\in G}d_{sup}(f\circ v\circ \beta_{\varphi\circ\rho(g)},\alpha_g\circ f\circ v)\leq 2(\varepsilon_1+\varepsilon_2).$ 
On the other hand, it is clear that the map $f\circ v: Y\to X$ is an $(\varepsilon_1+\varepsilon_2)$-isometry and therefore it is also a $2(\varepsilon_1+\varepsilon_2)$-isometry. Hence 
$$d_{GH,1}(\alpha,\beta)\leq 2(\varepsilon_1+\varepsilon_2)<2(d_{GH,1}(\alpha,\gamma)+d_{GH,1}(\gamma,\beta)+2\varepsilon), \mbox { for every }\varepsilon>0.$$
Therefore, $d_{GH,1}(\alpha,\beta)\leq 2(d_{GH,1}(\alpha,\gamma)+d_{GH,1}(\gamma,\beta)).$
\end{proof}
\begin{definition}
An action $\alpha$ of a topological group $G$ on a metric space $X$ is \textit{strongly GH-stable} if for every $\varepsilon>0$, there is $\delta>0$ such that for every continuous action $\beta$ of a topological group $H$ on a metric space $Y$ with $d_{GH,1}(\alpha,\beta)<\delta$ there are a continuous $\varepsilon$-isometry $h:Y\to X$ and a homomorphism $\rho:G\to H$ such that $\alpha_g\circ h=h\circ \beta_{\rho(g)}$ for every $g\in G$. 
\end{definition}
Before proving Theorem \ref{T-expansive and POTP implies GH-stability for two groups}, let us present a version of \cite[Lemma 2.10]{ChungLee} and \cite[Lemma 4.5]{KDD} for an expansive action of a general countable group on a proper metric space. We recall that a metric space $X$ is \textit{proper} if every closed ball is compact.
\begin{lem}
\label{L-expansiveness determines the topology}
Let $\alpha$ be an expansive action of a countable group $G$ on a proper metric space $(X,d)$ and let $c$ be an expansive constant of $\alpha$. Then, for every $x\in X$ and every $\varepsilon>0$, there exists a non-empty finite subset $F$ of $G$ such that whenever $\sup_{g\in F}d(\alpha_gx,\alpha_gy)\leq c$, one has $d(x,y)<\varepsilon$.
\end{lem}
\begin{proof}
We will prove by contradiction. Fix $x\in X$. Assume that there exists an $\varepsilon>0$ such that for every non-empty finite subset $F$ of $G$, there exists $y_F\in X$ such that $\sup_{g\in F} d(\alpha_gx,\alpha_gy_F)\leq c$ and $d(x,y_F)\geq\varepsilon$. Choose a sequence of finite subsets $F_n$ of $G$ such that $ \{e_G\}\subset F_1\subset F_2\subset \cdots$ and $G=\bigcup_{n\in \N}F_n$. Then for every $n\in \N$, there exists $y_n\in X$ such that $\sup_{g\in F_n} d(\alpha_gx,\alpha_gy_n)\leq c$ and $d(x,y_n)\geq\varepsilon$. As $X$ is proper, after taking a subsequence, we can assume that $y_n\to y$. Then we have $d(\alpha_gx,\alpha_gy)\leq c$ for all $g\in G$ and $d(x,y)\geq\varepsilon$ which contradicts expansiveness of $\alpha$.
\end{proof}
Now we are ready to prove Theorem \ref{T-expansive and POTP implies GH-stability for two groups}.
\begin{proof}[Proof of Theorem \ref{T-expansive and POTP implies GH-stability for two groups}]
(1) Let $\varepsilon>0$ with $\varepsilon<c$ and take $0<\varepsilon_1<\varepsilon/4$. As $\alpha$ has POTP there exist $\delta>0$ and a finite subset $S$ of $G$ such that every $(\delta,S)$-pseudo-orbit is $\varepsilon$-traced by some point $x\in X$. We can choose $\delta<\varepsilon_1$. Let $\beta$ be a continuous action of a topological group $H$ on a metric space $Y$ such that $d_{GH,1}(\alpha,\beta)<\delta$. Then there are a $\delta$-isometry $u:Y\to X$ and a homomorphism $\rho:G\to H$ such that $d_{sup}(u\circ \beta_{\rho(t)},\alpha_t\circ u)<\delta$ for every $t\in G$. Hence, for every $y\in Y$ and every $s\in S,t\in G$, one has 
$$d_X(\alpha_s(u(\beta_{\rho(t)}y)),u(\beta_{\rho(st)}y))= d_X(\alpha_s\circ u(\beta_{\rho(t)}y),u\circ\beta_{\rho(s)}(\beta_{\rho(t)}y))<\delta. $$  

Thus for every $y\in Y$, $\{u(\beta_{\rho(t)}y)\}_{t\in G}$ is a $(\delta,S)$-pseudo-orbit for $\alpha$. Hence, there exists $x\in X$ such that $d_X(\alpha_tx,u(\beta_{\rho(t)}y))<\varepsilon_1$ for every $t\in G$. By expansiveness of $\alpha$ and the choice of $\varepsilon_1$, it follows from Remark \ref{R-unique traced of pseudo-orbit} that such an $x$ is unique denoted by $h(y)$. Then we get the map $h:Y\to X$ satisfying 
$$d_X(\alpha_t h(y),u(\beta_{\rho(t)}y))< \varepsilon_1, \mbox{ for all } y\in Y, t\in G \mbox{  }(*).$$

In particular, we have $\sup_{y\in Y}d_X(h(y),u(y))\leq \varepsilon_1$. Therefore, for every $y_1,y_2\in Y$, we have 
\begin{eqnarray*}
|d_X(h(y_1),h(y_2))-d_Y(y_1,y_2)|&\leq &|d_X(h(y_1),h(y_2))-d_X(u(y_1),u(y_2))|+\\
&&|d_X(u(y_1),u(y_2))-d_Y(y_1,y_2)|\\
&\leq& |d_X(h(y_1),h(y_2))-d_X(h(y_1),u(y_2))|+ \\
&&|d_X(h(y_1),u(y_2))-d_X(u(y_1),u(y_2))|+\delta\\
&\leq & d_X(h(y_2),u(y_2))+d_X(h(y_1),u(y_1))+\delta\\
&\leq & 2\varepsilon_1+\delta<\varepsilon.
\end{eqnarray*}

On the other hand, one has 
$$d_H(X,h(Y))\leq d_H(X,u(Y))+d_H(u(Y),h(Y))\leq \delta+\varepsilon_1<\varepsilon.$$

Therefore the map $h:Y\to X$ is an $\varepsilon$-isometry. 

Now we will prove that $\alpha_t\circ h=h\circ \beta_{\rho(t)}$ for every $t\in G$. 
By (*), for every $y\in Y$ and $t_1,t_2\in G$,  one has 
\begin{align*}
&d_X(\alpha_{t_1}h(\beta_{\rho(t_2)}y),u\beta_{\rho(t_1t_2)}y))=d_X(\alpha_{t_1}h(\beta_{\rho(t_2)}y),u\beta_{\rho(t_1)}(\beta_{\rho(t_2)}y))<\varepsilon_1 \mbox{ and}\\
&d_X(\alpha_{t_1}\alpha_{t_2}h(y),u\beta_{\rho(t_1t_2)}y)=d_X(\alpha_{t_1t_2}h(y),u\beta_{\rho(t_1t_2)}y)<\varepsilon_1.
\end{align*}

Hence applying Remark \ref{R-unique traced of pseudo-orbit}, we get $\alpha_{t_2}\circ h=h\circ \beta_{\rho(t_2)}$ for every $t_2\in G$.

(2) Let $\varepsilon_2>0$. As $c$ is an expansive constant of $\alpha$, applying Lemma \ref{L-expansiveness determines the topology}, there exists a non-empty finite subset $A$ of $G$ such that whenever $\sup_{t\in A}d_X(\alpha_tx,\alpha_ty)\leq c$ one has $d_X(x,y)<\varepsilon_2$. As $Y$ is compact and $\beta$ is a continuous action, we can choose $\delta_1>0$ such that for every $y_1,y_2\in Y$ with $d_Y(y_1,y_2)<\delta_1$, one has $d_Y(\beta_{\rho(t)}y_1,\beta_{\rho(t)}y_2)<c/4$ for every $t\in A$. Then for every $y_1,y_2\in Y$ with $d_Y(y_1,y_2)<\delta_1$ and $t\in A$, applying (*) we get 
\begin{eqnarray*}
d_X(\alpha_th(y_1),\alpha_th(y_2))&\leq& d_X(\alpha_th(y_1),u\beta_{\rho(t)}(y_1))+d_X(u\beta_{\rho(t)}(y_1),u\beta_{\rho(t)}(y_2))+\\
&&+d_X(u\beta_{\rho(t)}(y_2),\alpha_th(y_2))\\
&< &2\varepsilon_1+d_Y(\beta_{\rho(t)}y_1,\beta_{\rho(t)}y_2)+\delta\\
&<&c/2+c/4+c/4=c.
\end{eqnarray*}

Hence $d_X(h(y_1),h(y_2))<\varepsilon_2$ for $y_1,y_2\in Y$ with $d_Y(y_1,y_2)<\delta_1$, showing that $h$ is continuous.
\end{proof}

\begin{remark}
It would be interesting to inquire whether or not we have a version of Theorem \ref{T-expansive and POTP implies GH-stability for two groups} for actions of a general topological group $G$. If the group $G$ is not countable, the Definition \ref{D-expansivity} is not good to define expansiveness for an action of $G$. For example, in the case $G=\R$, if we define expansiveness for a continuous action $\alpha$ of $G$ on a compact metric space $X$ as that there exists a constant $c>0$ such that $\sup_{t\in \R}d(\alpha_tx,\alpha_t y)<c$ implies $x=y$, then no such actions exist \cite[page 181]{BW}.  
\end{remark}
Now we illustrate some examples of expansive actions having POTP.
\begin{example}
Let $A\in GL_n(\R)$ and define the map $T:\T^n\to \T^n, x\mapsto Ax \mbox{ (mod } \Z^n)$. If $A$ does not have eigenvalues of modulus 1 then $T$ is expansive and has POTP. 
\end{example}
\begin{example}
We endow $\R^n$ with the usual Euclidean metric. Let $A\in GL_n(\R)$ and define the map $T:\R^n\to \R^n, x\mapsto Ax $. If $A$ does not have eigenvalues of modulus 1 then $T$ is expansive and has POTP \cite[Lemma 2.2.33,Theorem 2.3.15]{AH}. 
\end{example}
\begin{example}
Let $G$ be a finitely generated virtually nilpotent group, i.e. $G$ has a nilpotent subgroup with finite index. Let $\alpha$ be an action of $G$ on a compact metric space $X$. If there is some $g\in G$ such that $\alpha_g$ is expansive and has POTP then the action $\alpha$ is expansive and has POTP \cite[Lemma 2.13]{ChungLee}.

Combining with the previous example, we see that if we choose a finitely generated nilpotent subgroup $G$ of $GL_n(\R)$ such that there exists $g\in G$ which does not have eigenvalues of modulus 1, then the natural action of $G$ on $\T^n$ is expansive and has POTP.
\end{example}
\begin{example}
Let $G$ be a countable group. Then every subshift $G\curvearrowright X$, where $X\subset A^G$ for some finite set $A$, is expansive. Furthermore, the subshift $G\curvearrowright X$ has POTP if and only if it is of finite type \cite[Theorem 3.2]{ChungLee}.
\end{example}
\begin{example}
Let $G$ be a countable group. We denote by $\Z G$ the group ring of $G$ with coefficients in $\Z$. It consists of finitely supported $\Z$-valued functions
$f$
on $G$, which we shall write as $\sum_{s\in G}f_ss$. The algebraic structure of $\Z G$ is defined by $(\sum_{s\in G} f_s s)+(\sum_{s\in G} g_s s)=\sum_{s\in G}(f_s+g_s)s$ and $(\sum_{s\in G}f_ss)(\sum_{s\in G}g_ss)=\sum_{s\in G}(\sum_{t\in G}f_tg_{t^{-1}s})s$.

We denote by $\ell^1(G)$ the Banach algebra of all absolutely summable $\R$-valued functions on $G$, equipped with
the $\ell^1$-norm $\|\cdot \|_1$. We shall write $f\in \ell^1(G)$ as $\sum_{s\in G}f_ss$. Note that $\ell^1(G)$ has an involution $f\mapsto f^*$ defined by
$(\sum_{s\in G}f_ss)^*=\sum_{s\in G}f_ss^{-1}$.
\iffalse
For each $k\in \N$, we endow $\R^k$ the supremum norm $\|\cdot \|_\infty$. For each $1\le p\le +\infty$, we endow
$\ell^p(G, \R^k)=(\ell^p(G))^k$ with the $\ell^p$-norm
\begin{align} \label{E-infinity norm}
\| (f_1, \dots, f_k)\|_p=\|G\ni s\mapsto \|(f_1(s), \dots, f_k(s))\|_\infty\|_p.
\end{align}
We shall write elements of $(\ell^p(G))^k$ as row vectors.

The algebraic structures of $\Z G$ and $\ell^1(G)$ also extend to some other situations naturally. For example, $(\R/\Z)^G$ becomes
a right $\Z G$-module naturally.
\fi

For every $k\in \N$, we denote by $M_{k}(\ell^1(G))$ the space of all $k\times k$ matrices with entries in $\ell^1(G)$. 
The involution of $\ell^1(G)$ also extends naturally to an isometric linear map on $M_{k}(\ell^1(G))$ given
by
$$ (f_{i, j})_{1\leq i,j\le k}^*:=(f_{j, i}^*)_{1\le i,j\le k}.$$
%By an {\it algebraic action} of $G$, we mean an action of $G$ on a compact abelian group by automorphisms.

For a locally compact abelian group $X$, we denote by $\widehat{X}$ its Pontryagin dual.
% Then for any compact abelian group $X$, there is a natural one-to-one correspondence between algebraic actions of $G$ on $X$  and actions of $G$ on $\widehat{X}$ by automorphisms. There is also a natural one-to-one correspondence between the latter and left $\Z G$-module structure on $\widehat{X}$.
%Thus, when we have an algebraic action of $G$ on $X$, we shall talk about the left $\Z G$-module $\widehat{X}$. And when we have a left $\Z G$-module $W$, we shall treat $W$ as a discrete abelian group and talk about the algebraic action of $G$ on $\widehat{W}$.

Note that for each $k\in \N$, we may identify the Pontryagin dual $\widehat{(\Z G)^k}$  of $(\Z G)^k$ with $((\R/\Z)^k)^G=((\R/\Z)^G)^k$ naturally. Under this identification,
the canonical action of $G$ on $\widehat{(\Z G)^k}$ is just the left shift action on $((\R/\Z)^k)^G$. 
%If $J$ is a left $\Z G$-submodule of $(\Z G)^k$, then $\widehat{(\Z G)^k/J}$ is identified with $$ \{(x_1, \dots, x_k)\in ((\R/\Z)^G)^k: x_1g_1^*+\cdots+x_kg_k^*=0_{(\R/\Z)^G}, \mbox{ for all } (g_1, \dots, g_k)\in J\}.$$
For $A\in M_k(\Z G)$, we denote $X_A:=\widehat{(\Z G)^k/(\Z G)^kA}$ then $$X_A=\{x\in ((\R/\Z)^k)^G:(xA^*)_g=0_{\R^k/\Z^k}, \mbox{ for every }g\in G\}.$$

Let $\alpha_A$ be the natural left action of $G$ on $X_A$, i.e. $\alpha_{A,g}((x)_{h\in G})=(x_{g^{-1}h})_{h\in G}$, for every $g\in G$. If $A$ is invertible in $M_k(\ell^1(G))$ then $\alpha_A$ is expansive \cite[Lemma 3.7]{ChungLi} and has POTP \cite[Theorem 3.4]{Mey}.
\end{example}
\begin{example}
Let $G$ be the Baumslag-Solitar group $\langle a,b|ba=a^nb \rangle$, where $n\geq 2$. Let $\lambda>n$ and consider the action $\alpha$ of $G$ on $\R^2$ generated by two maps $\alpha_a(x):=Ax$, and $\alpha_b(x):=Bx$, where \[A=
\begin{bmatrix}
    1  &  0      \\
    1  &  1      
\end{bmatrix}, 
B= 
\begin{bmatrix}
    \lambda  &  0      \\
    0  &  n\lambda      
\end{bmatrix} 
\] 
As $BA=A^nB$, the action $\alpha$ is well defined. Since the map $\alpha_b$ is expansive the action $\alpha$ is also. And the POTP of $\alpha$ is from \cite[Theorem 2]{OT}.
\end{example}
The next two examples illustrate Theorem \ref{T-expansive and POTP implies GH-stability for two groups}.
\begin{example}
\label{E-GH distance 1}
Let $(X,d_X)$ and $(Y,d_Y)$ be metric spaces and the maps $h:Y\to X$, $f:X\to Y$ as in Example \ref{E-GH distance}. Let $A=\left(\begin{array}{cc} 1 & 2\\ 3 & 4 \end{array}\right)\in GL_2(\R)$ and $G$ be the group generated by $A$. Then $G$ is isomorphic to $\Z$. Let $\rho:G\to \Z^2$ be the homomorphism defined by $\rho(A^n)=(n,0)$ for every $n\in \Z$. Let $\alpha$ be the natural action of $G$ on $X=\T^2$. Let $\gamma$ be an arbitrary action of $\Z^2$ on $S^1_{1/n}\times S^1_{1/n}$ and let $\bar{\alpha}$ be an action of $\Z^2$ on $\T^2$ such that $\bar{\alpha}_{(1,0)}=\alpha_A$. Let $\bar{\beta}$ be the product action of $\Z^2$ on $(Y,d_Y)$ induced from the actions $\gamma$ and $\bar{\alpha}$. Then for every $n\in \Z, s\in S^1_{1/n}\times S^1_{1/n},x\in \T^2$, we have $\alpha_{A^n}(h(s,x))=\alpha_{A^n}(x)=\bar{\alpha}_{(n,0)}(x)=h(\bar{\beta}_{\rho(A^n)}(s,x)).$ Because $h:Y\to X$ is a $\frac{\sqrt{2}\pi}{n}$-isometry, we get that $d_{GH,1}(\alpha,\bar{\beta})\leq \frac{\sqrt{2}\pi}{n}$.

Now let $\rho_1:\Z^2\to G$ be a homomorphism such that $\rho_1((1,0))=A$. Let $\gamma$ be an arbitrary action of $\Z^2$ on $S^1_{1/n}\times S^1_{1/n}$ and let $\tilde{\alpha}$ be an action of $\Z^2$ on $\T^2$ defined by $\tilde{\alpha}_{g_1}:=\alpha_{\rho_1(g_1)}$ for every $g_1\in \Z^2$. Let $\tilde{\beta}$ be the product action of $\Z^2$ on $(Y,d_Y)$ induced from the actions $\gamma$ and $\tilde{\alpha}$.
Then for every $n\in \Z, s\in S^1_{1/n}\times S^1_{1/n},x\in X$, we have 
$$h(\tilde{\beta}_{\rho(A^n)}(s,x))=\tilde{\alpha}_{\rho(A^n)}x=\tilde{\alpha}_{(n,0)}x=\alpha_{\rho_1((n,0))}x=\alpha_{A^n}x=\alpha_{A^n}(h(s,x)).$$
On the other hand, for every $g_1\in \Z^2,x\in X$, we have 
\begin{eqnarray*}
d_Y(\tilde{\beta}_{g_1}f(x),f(\alpha_{\rho_1(g_1)}(x))&=&d_Y(\tilde{\beta}_{g_1}((\bar{s},\bar{s},x)),(\bar{s},\bar{s},\alpha_{\rho_1(g_1)}(x))\\
&=&d_Y(\gamma_{g_1}(\bar{s},\bar{s}),\tilde{\alpha}_{g_1}(x)), (\bar{s},\bar{s},\alpha_{\rho_1(g_1)}(x))\\
&=&d_Y(\gamma_{g_1}(\bar{s},\bar{s}),\alpha_{\rho_1(g_1)}(x)), (\bar{s},\bar{s},\alpha_{\rho_1(g_1)}(x))\\
&\leq & \sqrt{2\diam^2 S^1_{1/n}} = \frac{\sqrt{2}\pi}{n}.
\end{eqnarray*}
As $f:Y\to X$ is a $\frac{\sqrt{2}\pi}{n}$-isometry, we get that $d_{GH,2}(\alpha,\tilde{\beta})\leq \frac{\sqrt{2}\pi}{n}$. Because $A$ does not have eigenvalues of modulus 1 we know that $\alpha$ is expansive and has POTP.
\end{example}
\begin{example}
If we replace $(X,d_X)$  by $\R^2$ with the usual Euclidean metric and use the same arguments as in Example \ref{E-GH distance 1} then we can construct an action $\alpha$ of $\Z$ on $\R^2$, which is expansive and has POTP, and two actions $\bar{\beta},\tilde{\beta}$ of $\Z^2$ on $Y=S^1_{1/n}\times S^1_{1/n}\times X$
such that $$\max\{d_{GH,1}(\alpha,\bar{\beta}),d_{GH,2}(\alpha,\tilde{\beta})\} \leq \frac{\sqrt{2}\pi}{n}.$$ 
\end{example}

\section{Actions on Wasserstein spaces}
As in this section we deal with induced actions on Wasserstein spaces, we need GH-approximations be measurable. We first prove that for certain actions we can replace eGHAs by measurable ones. 

The following lemma should be well known, however we have not found it in the literature. Therefore, for completeness we give a simple proof here which follows from the idea of the proof of \cite[Lemma 3.5]{Shioya}.
\begin{lem}
\label{L-measurable GHA}
Let $(X,d_X)$ and $(Y,d_Y)$ be Polish metric spaces. Let $\varepsilon>0$ and $f:X\to Y$ be an $\varepsilon$-GH approximation.  Let $f':Y\to X$ be the inverse $3\varepsilon$-GH approximation of $f$ as in Definition \ref{D-inverse GHA}. Then 
\begin{enumerate}
\item
There exists a $5\varepsilon$-GH approximation $f_1:X\to Y$ such that $f_1$ is measurable and $d_Y(f(x),f_1(x))\leq 2\varepsilon$ for every $x\in X$.
\item There exists a $9\varepsilon$-GH approximation $f'_1:Y\to X$ such that $f'_1$ is measurable, $\sup_{y\in Y}d_X(f'(y),f'_1(y))\leq 2\varepsilon$, $\sup_{x\in X}d_X(x,(f'_1\circ f)(x))\leq 4\varepsilon$ and $\sup_{y\in Y}d_Y(y,(f\circ f'_1)(y))\leq 4\varepsilon$.
\end{enumerate}
\end{lem}
\begin{proof}
(1) As $X$ is separable, there exists a countable dense subset $\{x_n\}_{n\in \N}$ of $X$. We put $B_1:=B'_{\varepsilon}(x_1) \mbox{ and } B_{n+1}:=B'_{\varepsilon}(x_{n+1})\setminus\bigcup_{j=1}^nB'_{\varepsilon}(x_j), \mbox{ for } n\geq 1.$
Then $\{B_n\}_{n\in\N}$ is a disjoint covering of $X$ and $B_n$ is measurable for every $n\in \N$. For every $x\in X$, there exists a unique $n$ such that $x\in B_n$. We define the map $f_1:X\to Y$ by $f_1(x):=f(x_n)$ where $x\in B_n$. Then $f_1$ is measurable. For every $x\in B_n$, 
we have $d_Y(f(x),f_1(x))=d_Y(f(x),f(x_n))\leq d_X(x,x_n)+\varepsilon\leq 2\varepsilon$. Therefore, for every $x,x'\in X$, we obtain 
\begin{eqnarray*}
|d_Y(f_1(x),f_1(x'))-d_X(x,x')|&\leq & |d_Y(f_1(x),f_1(x'))-d_Y(f_1(x),f(x'))|+\\
&+&|d_Y(f_1(x),f(x'))-d_Y(f(x),f(x')|+|d_Y(f(x),f(x'))-d_X(x,x')|\\
&\leq &d_Y(f_1(x'),f(x'))+d_Y(f_1(x),f(x))+\varepsilon\\
&\leq & 5\varepsilon.
\end{eqnarray*}
On the other hand, for every $y\in Y$, there exists $x\in X$ such that $d_Y(y,f(x))\leq \varepsilon$ and hence $d_Y(y,f_1(x))\leq d_Y(y,f(x))+d_Y(f(x),f_1(x))\leq 3\varepsilon$. Therefore, $f_1$ is a measurable $5\varepsilon$-GH approximation from $X$ to $Y$.

(2) Similarly to (1), we can find a $9\varepsilon$-GH approximation $f'_1:Y\to X$ such that $f'_1$ is measurable, $\sup_{y\in Y}d_X(f'(y),f'_1(y))\leq 2\varepsilon$. For every $x\in X$, $$d_X(x,f_1'\circ f(x))\leq d_X(x,f'\circ f(x))+d_X(f'\circ f(x),f'_1\circ f(x))\leq 2\varepsilon+2\varepsilon=4\varepsilon.$$
And for every $y\in Y$, $$d_Y(y,f\circ f'_1(y))\leq d_Y(y,f\circ f'(y))+d_Y(f\circ f'(y),f\circ f'_1(y))\leq \varepsilon+d_Y(f'(y),f'_1(y))+\varepsilon\leq 4\varepsilon.$$
\end{proof}
\begin{lem}
\label{L-measurable eGHA}
\begin{enumerate}
\item Let $\alpha$ and $\beta$ be actions of topological groups $G$ and $H$ on Polish metric spaces $(X,d_X)$ and $(Y,d_Y)$, respectively. Let $\varepsilon>0$ and $(\rho,f):\alpha \to \beta$ be an $\varepsilon$-eFGHA. Assume that $\beta$ is an isometric action. Then there exists a $5\varepsilon$-eFGHA $(\rho,f_1):\alpha\to \beta$ such that $f_1$ is measurable.
\item Let $\alpha$ and $\beta$ be actions of a finitely generated group $G$ on Polish metric spaces $(X,d_X)$ and $(Y,d_Y)$ respectively, and let $S$ be a finitely generating subset of $G$. Let $\varepsilon>0$ and choose $\varepsilon'>0$ such that for every $y,y'\in Y$ with $d_Y(y,y')<3\varepsilon'$ one has $d_Y(\beta_s(y),\beta_s(y'))<\varepsilon$ for every $s\in S$. Let $f:G\curvearrowright X\to G\curvearrowright Y$ be an $(\varepsilon',S)$-GH approximation. Then there exists a $(5\varepsilon,S)$-GH approximation $f_1:G\curvearrowright X\to G\curvearrowright Y$ such that $f_1$ is measurable.
\end{enumerate}
\end{lem}
\begin{proof}
(1) From Lemma \ref{L-measurable GHA}, there exists a $5\varepsilon$-GH approximation $f_1:X\to Y$ such that $f_1$ is measurable and $d_Y(f(x),f_1(x))\leq 2\varepsilon$ for every $x\in X$. On the other hand, for every $x\in X$, $g\in G$, one has 
\begin{eqnarray*}
d_Y(f_1(\alpha_g(x)),\beta_{\rho(g)}(f_1(x))&\leq &d_Y(f_1(\alpha_g(x)),f(\alpha_g(x)))+d_Y(f(\alpha_g(x))),\beta_{\rho(g)}(f(x)))+\\
&&+d_Y(\beta_{\rho(g)}(f(x)),\beta_{\rho(g)}(f_1(x)))\\
&\leq &2\varepsilon+\varepsilon+d_Y(f(x),f_1(x))\\
&\leq & 5\varepsilon.
\end{eqnarray*}
(2) The proof is similar to that for (1).
\end{proof}
%From now on, we always assume that all Gromov-Hausdorff approximations are measurable.
Using the idea in the proof of \cite[Proposition 4.1]{LV}, we get the following lemma.
\begin{lem}
\label{L-GH approximations for Wasserstein spaces}
Let $\alpha_1$ and $\alpha_2$ be actions of a topological group $G$ on compact metric spaces $(X_1,d_1)$ and $(X_2,d_2)$, respectively. If $f:X_1\to X_2$ is an $\varepsilon$-measurable GH approximation from $\alpha_1$ to $\alpha_2$ then for every $p\geq 1$, the map $f_*:P_p(X_1)\to P_p(X_2)$ is an $\tilde{\varepsilon}$-measurable GH approximation from $(\alpha_1)_*$ and $(\alpha_2)_*$, where $\tilde{\varepsilon}=8\varepsilon+(9p(\diam(X_1)^{p-1}+\diam(X_2)^{p-1})\varepsilon)^{1/p}$. 
\end{lem}
\begin{proof}
Let $\mu_1,\mu_1'\in P_p(X_1)$ and let $\pi_1\in \Opt_p(\mu_1,\mu_1')$. Then $(f\times f)_*\pi_1\in \prod(f_*\mu_1,f_*\mu_1')$ and hence
\begin{eqnarray*}
W_p^p(f_*\mu_1,f_*\mu_1')&\leq& \int_{X_2\times X_2}d_2^p(x_2,y_2)d((f\times f)_*\pi_1)(x_2,y_2) \\
&=&\int_{X_1\times X_1}d_2^p(f(x_1),f(y_1))d\pi_1(x_1,y_1).
\end{eqnarray*} 
As the function $h(x)=x^p, x\geq 0$ has $h'(x)=px^{p-1}$, we have for every $x,y\geq 0$, 
$$|x^p-y^p|\leq p|x-y|\max\{x^{p-1},y^{p-1}\}\leq p|x-y|(x^{p-1}+y^{p-1}).$$
Therefore, for every $x_1,y_1\in X$,
\begin{align*}
|d_2^p(f(x_1),f(y_1))&-d_1^p(x_1,y_1)|\leq \\
&\leq p|d_2(f(x_1),f(y_1))-d_1(x_1,y_1)|(d_2^{p-1}(f(x_1),f(y_1))+d_1^{p-1}(x_1,y_1)).
\end{align*}
And hence,  
$$
|d_2^p(f(x_1),f(y_1))-d_1^p(x_1,y_1)|\leq pM |d_2(f(x_1),f(y_1))-d_1(x_1,y_1)|\leq  pM\varepsilon,
$$
where $M=\diam(X_1)^{p-1}+\diam(X_2)^{p-1}$. It follows that 
$$W_p^p(f_*\mu_1,f_*\mu_1')\leq W_p^p(\mu_1,\mu_1')+pM\varepsilon,$$
and hence $W_p(f_*\mu_1,f_*\mu_1')\leq (W_p^p(\mu_1,\mu_1')+pM\varepsilon)^{1/p}\leq W_p(\mu_1,\mu_1')+(pM\varepsilon)^{1/p}.$

Let $f':X_2\to X_1$ be the measurable GH-approximate inverse of $f$ as in Lemma \ref{L-measurable GHA}. Then $f'$ is a $9\varepsilon$-Gromov-Hausdorff approximation from $X_2$ to $X_1$, and $$\sup_{x_1\in X_1}d_1(x_1,f'\circ f(x_1))\leq 4\varepsilon \mbox{ and } \sup_{x_2\in X_2}d_2(x_2,f\circ f'(x_2))\leq 4\varepsilon.$$ Applying the same process as above we get 
$$W_p(f'_*(f_*\mu_1),f'_*(f_*\mu'_1))\leq W_p(f_*\mu_1,f_*\mu_1')+(9pM\varepsilon)^{1/p}.$$
As $\sup_{x_1\in X_1}d_1(x_1,f'\circ f(x_1))\leq 4\varepsilon $, applying Lemma \ref{L-contraction of Wasserstein distance}, we get 
$$W_p((f'\circ f)_*\mu_1,\mu_1)\leq 4\varepsilon \mbox{ and } W_p((f'\circ f)_*\mu'_1,\mu'_1)\leq 4\varepsilon.$$
Therefore, 
\begin{eqnarray*}
W_p(\mu_1,\mu_1')&\leq &W_p(\mu_1, f'_*(f_*\mu_1))+W_p(f'_*(f_*\mu_1),f'_*(f_*\mu'_1))+W_p(f'_*(f_*\mu'_1),\mu'_1)\\
&\leq &8\varepsilon+ W_p(f_*\mu_1,f_*\mu_1')+(9pM\varepsilon)^{1/p}.
\end{eqnarray*}
On the other hand, given $\mu_2\in P_p(X_2)$, since $\sup_{x_2\in X_2}d_2(x_2,f\circ f'(x_2))\leq 4\varepsilon$, applying Lemma \ref{L-contraction of Wasserstein distance}, we get $W_p((f\circ f')_*\mu_2,\mu_2))\leq 4\varepsilon$. Therefore, $f_*:P_p(X_1)\to P_p(X_2)$ is an $\tilde{\varepsilon}$-GH approximation.

Finally, for every $\mu_1\in P_p(X_1),g\in G$, let $\pi\in\prod ((\alpha_{2,g}\circ f)_*\mu_1,(f\circ \alpha_{1,g})_*\mu_1))$. 
Since for every $g\in G$, $d_{sup}(f\circ \alpha_{1,g},\alpha_{2,g}\circ f)\leq \varepsilon$, applying Lemma \ref{L-contraction of Wasserstein distance} again we get 
\begin{eqnarray*}
W_p^p((\alpha_2)_{*,g}\circ f_*(\mu_1),f_*\circ (\alpha_1)_{*,g}(\mu_1))&=&W_p^p((\alpha_{2,g}\circ f)_*\mu_1,(f\circ \alpha_{1,g})_*\mu_1)\\
&\leq & \int_{X_2\times X_2}d_2^p(x_2,y_2)d\pi(x_2,y_2)\\
&=&\int_{X_1}d_1^p(\alpha_{2,g}\circ f(x_1),f\circ \alpha_{1,g}(x_1))d\mu_1(x_1)\\
&\leq&\varepsilon^p<\tilde{\varepsilon}^p.
\end{eqnarray*}
\end{proof}

\begin{proof}[Proof of Theorem \ref{T-equivariant Gromov-Hausdorff of actions on Wasserstein spaces}] the theorem now follows from Lemma \ref{L-GH approximations for Wasserstein spaces}.
\end{proof}

\begin{remark}
From Lemma \ref{L-measurable eGHA}, we see that if $\alpha_n$, $\alpha$ are isometric actions then the conclusion of Theorem \ref{T-equivariant Gromov-Hausdorff of actions on Wasserstein spaces} still holds for $d_{GH}$ instead of $d_{mGH}$. Similarly, if $\alpha_n$, $\alpha$ are continuous actions of a finitely generated group $G$ and $S$ is a finitely generating set of $G$, then the result is also true for $d_{GH,S}$.
%We also have another proof of \ref{T-equivariant Gromov-Hausdorff of actions on Wasserstein spaces} by applying 
\end{remark}

\iffalse
\begin{remark}
We also have another proof of \ref{T-Gromov-Hausdorff convergence of invariant measures for compact groups}] by applying \cite[Theorem 2.1]{Fu86}, \cite[Lemma 5.36]{LV} and lemma \ref{L-measurable eGHA}.
\end{remark}
\fi
Now before proving Theorem \ref{T-Gromov-Hausdorff convergence of invariant measures for amenable groups}, let us recall the definition of a locally compact amenable group via F{\o}lner's property \cite[Section 8.4]{EW}. Let $G$ be a locally compact group and let $\lambda$ be a left Haar measure of $G$. We say $G$ is \textit{amenable} if for every compact subset $K$ of $G$ and every $\varepsilon>0$, there exists a Borel subset $F$ of $G$ such that $0<\lambda(F)<\infty$ and 
$$\frac{\lambda(gF\Delta F)}{\lambda(F)}<\varepsilon, \mbox{ for every } g\in K,$$ 
where $\Delta$ denotes the symmetric difference of sets. For more details on amenable groups, see \cite{P,Pier}.

Compact groups, locally compact solvable groups, and locally compact groups of polynomial growth are standard examples of amenable groups \cite[Corollary 13.5]{Pier}, \cite[Proposition 0.13]{P}.

Let $G$ be a locally compact amenable group and assume that $G$ is $\sigma$-compact, i.e. $G=\bigcup_{n=1}^\infty V_n$, where $V_n$ is a compact subset of $G$ for every $n\in \N$. We can consider $G=\bigcup_{n=1}^\infty K_n$, where $K_n$ is a compact subset of $G$ and $K_n\subset K_{n+1}$, for every $n\in \N$. As $G$ is amenable, there exists a sequence of Borel subsets $\{F_n\}_{n\in\N}$ of $G$ such that for every $n\in \N$, $0<\lambda(F_n)<\infty$ and 
$$\frac{\lambda(gF_n\Delta F_n)}{\lambda(F_n)}<\varepsilon, \mbox{ for every } g\in K_n.$$ 
Therefore, $\lim_{n\to\infty}\frac{\lambda(gF_n\Delta F_n)}{\lambda(F_n)}=0$, for every $g\in G$. A such sequence $\{F_n\}_{n\in \N}$ is called a left F{\o}lner sequence of $G$.

Another characterization of amenability of a locally compact group $G$ is that every continuous action of $G$ on a compact metric space $X$ always has an invariant probability measure, i.e. $P^G(X)\neq \varnothing$ \cite[Theorem 5.4]{Pier}.

Let $\alpha$ and $\alpha_1$ be continuous actions of a locally compact, $\sigma$-compact amenable group $G$ on compact metric spaces $(X,d)$ and $(X_1,d_1)$, respectively. Let $\lambda$ be a left Haar measure of $G$ and let $\{F_n\}_{n\in \N}$ be a left F{\o}lner sequence of $G$.

 Fix $p>1$ and $\varepsilon>0$. Let $T=\{\mu_1,\cdots,\mu_N\}$ be an $\varepsilon$-net of $P_p^G(X)$ and let $f:X\to X_1$ be an $\varepsilon$-measurable GH approximation. As $P_p(X_1)$ is compact in the weak*-topology, passing to a subsequence if necessary, we assume that $\frac{1}{\lambda(F_n)}\int_{F_n}\alpha_1(g)_*(f_*\mu_1)d\lambda(g)\to \varphi_T(\mu_1)$ in the weak*-topology as $n\to \infty$. Do the same process for $\mu_2,\mu_3,\cdots$, we can define that for every $1\leq k \leq N$, $\varphi_T(\mu_k):=\lim_{n\to\infty}\frac{1}{\lambda(F_n)}\int_{F_n}\alpha_1(g)_*(f_*\mu_k)d\lambda(g)$. 
 
 Now we prove that $\varphi_T(\mu_k)\in P_p^G(X_1)$ for every $1\leq k\leq N$. Given $u\in C(X_1)$, $\mu_k\in T$, and $h\in G$, one has 
 \begin{align*}
 \int_{X_1}ud(\varphi_T(\mu_k))&=\lim_{n\to \infty}\frac{1}{\lambda(F_n)}\int_{F_n}\int_{X_1}u(\alpha_1(g)\circ f(x_1))d\mu_k(x_1)d\lambda(g),\\
  \int_{X_1}ud(\alpha_1(h)_*\varphi_T(\mu_k))&=\lim_{n\to \infty}\frac{1}{\lambda(F_n)}\int_{F_n}\int_{X_1}u(\alpha_1(hg)\circ f(x_1))d\mu_k(x_1)d\lambda(g)\\
  &=\lim_{n\to \infty}\frac{1}{\lambda(F_n)}\int_{hF_n}\int_{X_1}u(\alpha_1(g)\circ f(x_1))d\mu_k(x_1)d\lambda(g).
 \end{align*}
 Therefore, 
 \begin{eqnarray*}
 \left| \int_{X_1}ud(\alpha_1(h)_*\varphi_T(\mu_k))-\int_{X_1}ud(\varphi_T(\mu_k))\right|&\leq& \lim_{n\to \infty}\frac{1}{\lambda(F_n)}\int_{F_n\Delta hF_n}\int_{X_1}|u(\alpha_1(g)\circ f(x_1))|d\mu_k(x_1)d\lambda(g)\\
 &\leq &\lim_{n\to \infty}\frac{\lambda(F_n\Delta hF_n)}{\lambda(F_n)}\|u\|_\infty=0.
 \end{eqnarray*}
 
 Let $f':X_1\to X$ be the $9\varepsilon$-measurable GH approximation inverse of $f$ as in Lemma \ref{L-measurable GHA}. Similarly as above, we can also define that for every $1\leq k\leq N$,
 $$\varphi'_T(\varphi_T(\mu_k)):=\lim_{n\to\infty}\frac{1}{\lambda(F_n)}\int_{F_n}\alpha(g)_*(f'_*(\varphi_T(\mu_k)))d\lambda(g)\in P_p^G(X_1).$$
\begin{lem}
\label{L-GHA for actions of amenable groups}
Let $\alpha$ and $\alpha_1$ be isometric actions of a locally compact amenable, $\sigma$-compact group $G$ on compact metric spaces $(X,d)$ and $(X_1,d_1)$ respectively. Let $\varepsilon>0$ and let $f:X\to X_1$ be an $\varepsilon$-measurable GH-approximation from $\alpha_1$ to $\alpha_2$. Let $T=\{\mu_1,\cdots,\mu_N\}$ be an $\varepsilon$-net of $P_p^G(X)$. Then for every $p\geq 1$, the set $\varphi_T(T)$ is a $D(\varepsilon)$-net of $P_p^G(X_1)$ and
$$ |W_p(\varphi_T(\mu_i),\varphi_T(\mu_j))-W_p(\mu_i,\mu_j)|\leq D(\varepsilon), \mbox{ for every } \mu_i,\mu_j\in T,$$
where $D(\varepsilon)= 28\varepsilon+(9p(\diam(X_1)^{p-1}+\diam(X_2)^{p-1})\varepsilon)^{1/p}$.
%$\tilde{\varepsilon}=4\varepsilon+(3p(\diam(X_1)^{p-1}+\diam(X_2)^{p-1})\varepsilon)^{1/p}$. 
\end{lem}
\begin{proof}
Let $\mu_i,\mu_j\in T$ and let $\pi\in \Opt_p(\mu_i,\mu_j)$ then one has $$\lim_{n\to\infty}\frac{1}{\lambda(F_n)}\int_{F_n}\alpha_1(g)_*((f\times f)_*\pi)d\lambda(g)\in \prod(\varphi_T(\mu_i),\varphi_T(\mu_j)).$$
Therefore, 
\begin{eqnarray*}
W_p^p(\varphi_T(\mu_i),\varphi_T(\mu_j))&\leq &\int_{X_1\times X_1}d_1^p(x_1,y_1)d\big(\lim_{n\to\infty}\frac{1}{\lambda(F_n)}\int_{F_n}\alpha_1(g)_*((f\times f)_*\pi)d\lambda(g)\big)(x_1,y_1)\\
&=&\lim_{n\to\infty}\int_{X_1\times X_1}\frac{1}{\lambda(F_n)}\int_{F_n}d_1^p(x_1,y_1)d\big(\alpha_1(g)_*((f\times f)_*\pi)\big)(x_1,y_1)d\lambda(g)\\
&=&\lim_{n\to\infty}\frac{1}{\lambda(F_n)}\int_{X\times X}\int_{F_n}d^p(f(x),f(y))d\pi(x,y)d\lambda(g)\\
&\leq & \lim_{n\to\infty}\frac{1}{\lambda(F_n)}\int_{F_n}\int_{X\times X}(d^p(x,y)+pM\varepsilon)d\pi(x,y)d\lambda(g)\\
&\leq& W_p^p(\mu_i,\mu_j)+pM\varepsilon, 
\end{eqnarray*}
where $M=\diam(X_1)^{p-1}+\diam(X_2)^{p-1}$. Thus, 
\begin{align}
\label{P-GH approximation of actions of amenable groups}
W_p(\varphi_T(\mu_i),\varphi_T(\mu_j))\leq  W_p(\mu_i,\mu_j)+(pM\varepsilon)^{1/p}.
\end{align}

As $f':X_1\to X$ is a $9\varepsilon$-measurable GH approximation, similar as above, we get that $W_p(\varphi'_T(\varphi_T(\mu_i)),\varphi'_T(\varphi_T(\mu_j)))\leq  W_p(\varphi_T(\mu_i),\varphi_T(\mu_j))+(9pM\varepsilon)^{1/p}$, for every $i,j$.

On the other hand, for every $i$, one has 
\begin{eqnarray*}
\varphi'_T(\varphi_T(\mu_i))&=&\lim_{n\to\infty}\frac{1}{\lambda(F_n)}\int_{F_n}\alpha(g)_*(f'_*(\varphi_T(\mu_i)))d\lambda(g)\\
&=&\lim_{n\to\infty}\frac{1}{\lambda(F_n)}\int_{F_n}\alpha(g)_*(f'_*(\lim_{m\to\infty}\frac{1}{\lambda(F_m)}\int_{F_m}\alpha_1(h)_*(f_*\mu_i)d\lambda(h)))d\lambda(g)\\
&=&\lim_{m,n\to\infty}\frac{1}{\lambda(F_n)}\frac{1}{\lambda(F_m)}\int_{F_n}\alpha(g)_*(f'_*(\int_{F_m}\alpha_1(h)_*(f_*\mu_i)d\lambda(h)))d\lambda(g)\\
&=&\lim_{m,n\to\infty}\frac{1}{\lambda(F_n)}\frac{1}{\lambda(F_m)}\int_{F_m}\int_{F_n}\alpha(g)_*f'_*\alpha_1(h)_*f_*\mu_id\lambda(h)d\lambda(g).
\end{eqnarray*} 
Since $f$ is an $\varepsilon$-GH approximation from $\alpha$ to $\alpha_1$, for every $h\in G$, we get $$\sup_{x\in X}d_1(\alpha_1(h)\circ f(x),f\circ \alpha(h)(x))\leq \varepsilon.$$ 
In another side, because $\sup_{x\in X}d(x,f'(f(x))\leq 4\varepsilon$, and $f':X_1\to X$ is a $9\varepsilon$-GH approximation, for every $g,h\in G$, $x\in X$, one has 
\begin{eqnarray*}
d(\alpha(g)f'\alpha_1(h)f(x),\alpha(gh)(x))&=&d(f'\alpha_1(h)f(x),\alpha(h)(x))\\
&\leq& d(f'\alpha_1(h)f(x),f'f\alpha(h)(x))+d(f'f\alpha(h)(x),\alpha(h)(x)) \\
&\leq& d_1(\alpha_1(h)f(x),f\alpha(h)(x))+9\varepsilon+4\varepsilon\\
&\leq& 14\varepsilon
\end{eqnarray*}
Therefore, applying Lemma \ref{L-contraction of Wasserstein distance}, for every $g,h\in G$, $1\leq i\leq N$, we get $$W_p(\alpha(g)_*f'_*\alpha_1(h)_*f_*\mu_i,\mu_i)=W_p(\alpha(g)_*f'_*\alpha_1(h)_*f_*\mu_i,\alpha(gh)_*\mu_i)\leq 14\varepsilon.$$
For every $g,h\in G, 1\leq i\leq N$, let $\pi_{gh,i}\in \Opt_p(\mu_i,\alpha(g)_*f'_*\alpha_1(h)_*f_*\mu_i)$ then $$\int_{X\times X}d^p(x,y)d\pi_{gh,i}(x,y)\leq (14\varepsilon)^p \mbox{ and }$$
\begin{eqnarray*}
\pi_i:&=&\lim_{m,n\to\infty}\frac{1}{\lambda(F_n)}\frac{1}{\lambda(F_m)}\int_{F_m}\int_{F_n}\pi_{gh,i}d\lambda(h)d\lambda(g)\\
&\in& \prod\big(\mu_i,\lim_{m,n\to\infty}\frac{1}{\lambda(F_n)}\frac{1}{\lambda(F_m)}\int_{F_m}\int_{F_n}\alpha(g)_*f'_*\alpha_1(h)_*f_*\mu_id\lambda(h)d\lambda(g)\big)\\
&=&\prod\big(\mu_i,\varphi'_T(\varphi_T(\mu_i))\big).
\end{eqnarray*}
Therefore, 
\begin{eqnarray*}
W_p^p(\varphi'_T(\varphi_T(\mu_i)),\mu_i)&\leq& \int_{X\times X}d^p(x,y)d\pi_i(x,y)\\
&=&\int_{X\times X}d^p(x,y)d\big(\lim_{m,n\to\infty}\frac{1}{\lambda(F_n)}\frac{1}{\lambda(F_m)}\int_{F_m}\int_{F_n}\pi_{gh,i}d\lambda(h)d\lambda(g)\big)(x,y)\\
&=&\lim_{m,n\to\infty}\frac{1}{\lambda(F_n)}\frac{1}{\lambda(F_m)}\int_{X\times X}d^p(x,y)\int_{F_m}\int_{F_n}d\pi_{gh,i}(x,y)d\lambda(g)d\lambda(h)\\
&=&\lim_{m,n\to\infty}\frac{1}{\lambda(F_n)}\frac{1}{\lambda(F_m)}\int_{F_m}\int_{F_n}\int_{X\times X}d^p(x,y)d\pi_{gh,i}(x,y)d\lambda(g)d\lambda(h)\\
&\leq &\lim_{m,n\to\infty}\frac{1}{\lambda(F_n)}\frac{1}{\lambda(F_m)}\int_{F_m}\int_{F_n}(14\varepsilon)^pd\lambda(g)d\lambda(h)\\
&=&(14\varepsilon)^p.
\end{eqnarray*} 
And hence for every $\mu_i,\mu_j\in T$,
\begin{eqnarray*}
W_p(\mu_i,\mu_j)&\leq& W_p(\mu_i, \varphi'_T(\varphi_T(\mu_i)))+W_p(\varphi'_T(\varphi_T(\mu_i)),\varphi'_T(\varphi_T(\mu_j))+W_p(\varphi'_T(\varphi_T(\mu_j)),\mu_j)\\
&\leq& W_p(\varphi_T(\mu_i),\varphi_T(\mu_j))+28\varepsilon+(9pM\varepsilon)^{1/p}.
\end{eqnarray*}
Combining with (\ref{P-GH approximation of actions of amenable groups}) we get that $|W_p(\varphi_T(\mu_i),\varphi_T(\mu_j))-W_p(\mu_i,\mu_j)|\leq D(\varepsilon)$, for every $\mu_i,\mu_j\in T$.

Finally, we will prove that $\varphi_T(T)$ is a $D(\varepsilon)$-net of $P_p^G(X_1)$. Let $\mu\in P_p^G(X_1) $. There exists a subsequence $\{n_1\}$ such that $\lim_{n_1\to\infty}\frac{1}{\lambda(F_{n_1})}\int_{F_{n_1}}\alpha(g)_*(f'_*\mu)d\lambda(g)$ exists in the weak*-topology of $P_p(X)$, and we denote this limit as $\varphi'_{T\cup\{\mu\}}(\mu)$. Then there exists a subsequence $\{n_2\}$ of $\{n_1\}$ such that 
$$\lim_{n_2\to\infty}\frac{1}{\lambda(F_{n_2})}\int_{F_{n_2}}\alpha_1(g)_*(f_*\varphi'_{T\cup\{\mu\}}(\mu))d\lambda(g)$$
 exists in the weak*-topology of $P_p(X_1)$, and we denote this limit as $\varphi_{T\cup\{\mu\}}(\varphi'_{T\cup\{\mu\}}(\mu))$. Similar as above, we get
\begin{align*}
W_p(\varphi_{T\cup\{\mu\}}(\varphi'_{T\cup\{\mu\}}(\mu)),\mu)&\leq 14\varepsilon, \mbox{ and }\\
W_p(\varphi_T(\mu_k),\varphi_{T\cup\{\mu\}}(\varphi'_{T\cup\{\mu\}}(\mu))&\leq W_p(\mu_k,\varphi'_{T\cup\{\mu\}}(\mu))+(pM\varepsilon)^{1/p}, \mbox{ for every } 1\leq k\leq N. 
\end{align*}  
  As $T$ is an $\varepsilon$-net in $P_p^G(X)$, there exists $\mu_i\in T$ such that $W_p(\mu_i, \varphi'_{T\cup\{\mu\}}(\mu))\leq \varepsilon$. Then
\begin{eqnarray*}
W_p(\varphi_T(\mu_i),\mu)&\leq &W_p(\varphi_T(\mu_i),\varphi_{T\cup\{\mu\}}(\varphi'_{T\cup\{\mu\}}(\mu))+W_p(\varphi_{T\cup\{\mu\}}(\varphi'_{T\cup\{\mu\}}(\mu)),\mu)\\
&\leq & W_p(\mu_i,\varphi'_{T\cup\{\mu\}}(\mu))+(pM\varepsilon)^{1/p}+14\varepsilon\\
&\leq& 15\varepsilon+(pM\varepsilon)^{1/p}< D(\varepsilon).
\end{eqnarray*}
\end{proof}
\begin{proof}[Proof of Theorem \ref{T-Gromov-Hausdorff convergence of invariant measures for amenable groups}]
From Lemma \ref{L-closedness of isometric actions} we obtain that the action $\alpha$ is an isometric action. Then applying Lemma \ref{L-finite net implies closeness of GH}, Lemma \ref{L-measurable eGHA} and Lemma \ref{L-GHA for actions of amenable groups} we get the result.
\end{proof}
Now we illustrate examples satisfying our assumptions of Theorems \ref{T-equivariant Gromov-Hausdorff of actions on Wasserstein spaces}, \ref{T-Gromov-Hausdorff convergence of invariant measures for amenable groups} and Corollary \ref{C-uniquely ergodic}.
\begin{example}
\label{E-GH distance isometry}
We consider $S^1_r=\{z\in \C:|z|=r\}$ with the usual distance and $S^1=S^1_1$. For every $n\in \N$, let $X_n=S^1_{1/n}\times S^1$ with the canonical product metric $d_{X_n}$ and let $\alpha_n$ be the action of $S^1$ on $X_n$ defined by $\alpha_{n,a}(s,z)=(as,az)$ for every $s\in S^1_{1/n}, a,z\in S^1$. Let $\alpha$ be the action on $X=S^1$ by setting $\alpha_a(z)=az$ for every $a,z\in S^1$. Then $\alpha_n$, $\alpha$ are isometric actions. We define the map $h_n:X_n\to X$ by $h_n(s,z):=z$ for every $s\in S^1_{1/n},z\in X$. For every $n\in \N$, fix $\bar{s_n}\in S^1_{1/n}$, we define the map $f_n:X\to X_n$ by $f_n(z)=(\bar{s_n},z)$ for every $z\in X$. Then for every $n\in \N, s\in S^1_{1/n},z\in X, a\in S^1$, we have
$$h_n\alpha_{n,a}(s,z)=h_n(as,az)=az=\alpha_ah_n(s,z).$$
Similarly to the arguments in Example \ref{E-GH distance} we get that the maps $h_n:X_n\to X, f_n:X\to X_n$ are $\frac{\pi}{n}$-isometry. On the other hand, for every $n\in \N$, $a\in S^1$, $x\in X$, we also have
$$d_{X_n}(f_n(\alpha_a(x)),\alpha_{n,a}(f_n(x)))=d_{X_n}((\bar{s_n},ax),(a\bar{s_n},ax))\leq \frac{\pi}{n}.$$
Therefore $d_{GH}(\alpha_n,\alpha)\to 0$ as $n\to \infty$.
Now we will prove that $\alpha$ is uniquely ergodic. Let $\mu$ be a probability $S^1$-invariant measure of $\alpha$. Then for every $g\in S^1$ and every Borel subset $A$ of $X$ we have 
$\mu(\alpha_g^{-1}A)=\mu(A)$ and hence $\mu(g^{-1}A)=\mu(A)$. Therefore, $\mu$ must coincide with the normalized Haar measure of $X$.
\end{example}
\begin{example}
We choose $(X_n,d_{X_n}), (X,d)$ as in the previous example. Let $a\in S^1$ such that $a^m\neq 1$ for every $m\in \N$. Let $\alpha$ be the action of $\Z$ on $X$ defined by $\alpha_m(x)=a^mx$ for every $m\in \Z,x\in X$. For every $n\in \N$, let $\alpha_n$ be the action of $\Z$ on $X_n$ defined by $\alpha_{n,m}(s,z)=(a^ms,a^mz)$ for every $s\in S^1_{1/n}, z\in X, m\in \Z$. Then $\alpha_n$, $\alpha$ are isometric actions. Similar as Example \ref{E-GH distance isometry}, we have $d_{GH}(\alpha_n,\alpha)\to 0$ as $n\to \infty$. On the other hand, applying \cite[Theorem 6.18]{Walters82} we get that $\alpha$ is uniquely ergodic. 
\end{example}
\begin{example}
Let $X=S^1\times S^1$ with its canonical metric $d_X$. For every $n\in \N$ let $X_n=S^1_{1/n}\times X$ with its usual product metric $d_{X_n}$. Let $a_1,a_2,a_3\in S^1$ such that $a_1^m\neq 1, a_3^m\neq 1$ for every $m\in \N$ and $a_2^5=1$. We define the action $\alpha$ of $\Z$ on $X$ by $\alpha_{m}(s_2,s_3):=(a_2^ms_2,a_3^ms_3)$ for every $m\in \Z, s_2,s_3\in S^1$, and the action
 $\alpha_n$ of $\Z$ on $X_n$ by $\alpha_{n,m}(s_1,s_2,s_3):=(a_1^ms_1,a_2^ms_2,a_3^ms_3)$ for every $m\in \Z, s_1\in S^1_{1/n},s_2,s_3\in S^1$.
 Then $\alpha$, $\alpha_n$ are isometric actions and $\lim_{n\to\infty}d_{GH}(\alpha_n,\alpha)=0$. As the action $\alpha$ is not minimal, applying \cite[Theorem 6.20]{Walters82} we get that $\alpha$ is not uniquely ergodic.
\end{example}

\end{document}